\documentclass[12pt]{amsart}

\usepackage{amsmath, amssymb, amsthm}
\usepackage[T1]{fontenc}
\usepackage{comment}

\numberwithin{equation}{section}

\usepackage[margin=1in]{geometry}
\usepackage{subcaption}
\usepackage{graphicx}
\graphicspath{{figures/}}

\usepackage{enumitem}

\usepackage{hyperref}
	\hypersetup{colorlinks, breaklinks,
            	linkcolor = blue,
				urlcolor = blue,
				anchorcolor = blue,
				citecolor = blue}

\usepackage{cleveref}

\usepackage{mathtools}

\usepackage{multirow}

\DeclareMathOperator{\dist}{dist}

\DeclareMathOperator{\proj}{proj}

\let\Cap\relax
\DeclareMathOperator{\Cap}{Cap}
\DeclareMathOperator{\spt}{spt}

\newcommand{\Fav}{\operatorname{Fav}}

\newcommand{\C}{\mathbb{C}}

\newcommand{\E}{\mathbb{E}}

\newcommand{\N}{\mathbb{N}}
\renewcommand{\P}{\mathbb{P}}
\newcommand{\R}{\mathbb{R}}

\newcommand{\ii}{{\bf i}}
\newcommand{\jj}{{\bf j}}

\newcommand{\cC}{\mathcal{C}}
\newcommand{\cD}{\mathcal{D}}

\newcommand{\cH}{\mathcal{H}}
\newcommand{\cL}{\mathcal{L}}

\newcommand{\cS}{\mathcal{S}}

\usepackage{mathrsfs}

\renewcommand\epsilon{{\varepsilon}}

\newcommand{\wt}{\widetilde}

\usepackage{dsfont}

\usepackage{thmtools}
\declaretheorem[numberwithin=section,name=Theorem]{theorem}
\declaretheorem[sibling=theorem,name=Lemma]{lemma}
\declaretheorem[sibling=theorem,name=Proposition]{proposition}
\declaretheorem[sibling=theorem,name=Corollary]{corollary}

\declaretheorem[sibling=theorem,name=Remark,style=definition]{remark}
\declaretheorem[sibling=theorem,name=Problem,style=definition]{problem}

\newcommand{\discmap}[1]{\Phi_{#1}}

\newcommand{\fjump}{f_{\vec s}}
\newcommand{\fnojump}{g_{\vec s}}

\usepackage{xcolor}

\definecolor{alanorange}{RGB}{204, 112, 0} 

\title{Favard length of non-homogeneous random disc-like Cantor sets}

\author{Alan Chang, Damian D\k{a}browski, Giacomo Del Nin}


\begin{document}

\begin{abstract}
    The Favard length of a planar set is the average length of its orthogonal projections. We provide sharp estimates for the Favard length decay of a random variant of the non-homogeneous four corner Cantor sets introduced by Garnett in the 1970s. As a corollary, we obtain examples which show that Mattila's classical lower bound on the Favard length in terms of the Riesz capacity is sharp for all possible decay rates. We also obtain a new family of examples of sets with positive analytic capacity and zero Favard length.
\end{abstract}

\maketitle

\setcounter{tocdepth}{1}

\tableofcontents

\section{Introduction}

\subsection{Projections and Favard length}
For any $\theta\in [0,\pi]$ we denote the orthogonal projection $\proj_\theta:\R^2\to \R$ by 
\begin{equation*}
    \proj_\theta(x)=x\cdot(\cos\theta,\sin\theta).
\end{equation*}
Given a Borel set $E\subset\R^2$, its \emph{Favard length} is the average length of its orthogonal projections
\begin{equation*}
    \Fav(E)= \int_0^{\pi} |\proj_\theta(E)|\, d\theta,
\end{equation*}
where $|\cdot|$ stands for the $1$-dimensional Lebesgue measure. Favard length is also commonly called \emph{Buffon's needle probability}.

Recall that a Borel set $E\subset\R^2$ is called \emph{purely unrectifiable} if $\cH^1(E\cap\Gamma)=0$ for every rectifiable curve $\Gamma$. ($\cH^1$ denotes the $1$-dimensional Hausdorff measure, also referred to as ``length.'') By the Besicovitch projection theorem (e.g., \cite[Theorem 18.1]{mattila1995textbook}) every purely unrectifiable set $E$ of finite length satisfies $\Fav(E)=0$. Consequently, assuming $E$ is compact, we get
\begin{equation*}
    \Fav(E(r))\xrightarrow{r\to 0} 0,
\end{equation*}
where $E(r)$ denotes the closed $r$-neighbourhood of $E$, i.e.,
\[
E(r) = \{x \in \R^2 : \dist(x, E) \leq r\}.
\]
The \emph{Favard length problem}, posed in \cite{ps}, asks for estimating the rate of decay of $\Fav(E(r))$. Note that it may depend on the choice of the set $E$, and there are purely unrectifiable sets with arbitrarily slow Favard length decay \cite{wilson2017}.

One of the first general lower bounds on Favard length is due to Mattila. To state the bound, we first recall the Riesz $1$-energy of a measure $\mu$ on $\R^2$, defined by
\begin{align}
I_1(\mu) 
=
\iint \frac{d\mu (x) \, d\mu (y)}{|x-y|}.
\end{align}
Then we define the Riesz $1$-capacity of a compact set $E$ by
\begin{align}
\Cap_1(E) 
= 
\sup\{I_1(\mu)^{-1} : \mu \text{ probability measure supported on $E$}\}
.
\end{align}
Mattila {\cite[Corollary 3.3]{mattila1990}} proved that
\begin{align}\label{eq:mat-low-bd}
\Fav(E)
\gtrsim 
\Cap_1(E)
\qquad\text{for all compact sets $E \subset \R^2$}.
\end{align}
On the other hand, finding general upper bounds on Favard length is much harder. The only results in this direction we are aware of are due to Tao \cite{tao2009quantitative} and the second author \cite{dabrowski2024}. These upper bounds, which depend on quantifying pure unrectifiability, are very far from matching \eqref{eq:mat-low-bd}. At the same time, significantly better estimates have been shown for various self-similar sets.

The prototypical example is the \emph{four corner Cantor set}. Usually, it is described in terms of the corners of a $4 \times 4$ square grid. Here, we give a somewhat non-standard but equivalent description (see \Cref{fig:four corner Cantor discs}):

\begin{enumerate}
\item
Start with $K_0$ as the closed unit disc $B(0,1)$. 
\item Suppose $K_n$ has been constructed as the union of $4^n$ discs of radius $4^{-n}$. We will construct $K_{n+1}$ as follows: For each disc $D \subset K_n$, place four discs inside $D$ of radius $4^{-(n+1)}$, each internally tangent to $\partial D$, with their centers forming the vertices of an axis-aligned square. (Each of these four sub-discs will be referred to as the \emph{children} of $D$.) Let $K_{n+1}$ be the union of these $4^{n+1}$ discs.
\item 
The four corner Cantor set is $K = \bigcap_{n=0}^\infty K_n$.
\end{enumerate}

\begin{figure}[h]
\centering
\includegraphics[page=1,width=0.23\textwidth]{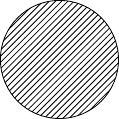}
\hfill
\includegraphics[page=2,width=0.23\textwidth]{standard_four_corner1.pdf}
\hfill
\includegraphics[page=3,width=0.23\textwidth]{standard_four_corner1.pdf}
\hfill
\includegraphics[page=4,width=0.23\textwidth]{standard_four_corner1.pdf}
\caption{Standard four corner Cantor set ($K_0, K_1, K_2, K_3$)}
\label{fig:four corner Cantor discs}
\end{figure}

The Besicovitch projection theorem implies $\Fav(K) = 0$. (Alternative proofs that $\Fav(K) = 0$ can be found in \cite{pss} and \cite[Chapter 10]{mattila2015textbook}.)
Peres and Solomyak  \cite{ps} gave an explicit, very slowly decaying upper bound for $\Fav(K_n)$ in terms of the iterated logarithm function. Then Nazarov, Peres, and Volberg \cite{npv} improved this significantly to a polynomial upper bound: $\Fav(K_n) \lesssim_{\epsilon} n^{-1/6+\epsilon}$ for all $\epsilon > 0$. Very recently, Marshall \cite{Marshall2025} further improved it to $\Fav(K_n) \lesssim_{\epsilon} n^{-1/5+\epsilon}.$ See also \cite{laba2010favard,blv, BondVolberg10, BondVolberg12, LabaMarshall22}
for upper bounds on the Favard length of other self-similar sets.  \cite{LabaMcDonaldTaylor2026} recently proved upper bounds on generalized Favard lengths associated with nonlinear projections in terms of the standard Favard length.

For lower bounds, it is not hard to show that $\Cap_1(K_n) \lesssim n$. (In fact, Bongers \cite{bongers2026} showed $\Cap_1(K_n) \approx n$.) Thus, Mattila's bound \eqref{eq:mat-low-bd} implies $\Fav(K_n) \gtrsim 1/n$. Bateman and Volberg \cite{bv} improved the lower bound to 
\begin{align}
\label{eq:bateman-volberg}
    \Fav(K_n) \gtrsim \frac{\log n}{n}
    .
\end{align} 
As shown in \cite[Appendix A]{chang-shmerkin-suomala}, this $(\log n)/n$ lower bound holds more generally for Cantor-type sets satisfying an ``alignment'' property.
On the other hand, Mattila's lower bound was shown to be sharp for certain grid-based \emph{random} Cantor sets in \cite{ps}, \cite{ps2005} and \cite{chang-shmerkin-suomala}.

For more information on the Favard length problem, see the survey articles \cite{laba2015, taylor2024}.

\subsection{Non-homogeneous disc-like Cantor sets}

We make two generalizations of the construction of the four corner Cantor set given above. 

\begin{enumerate}
    \item $K_0$ is a single disc of radius $r_0$.
    \item Suppose $K_n$ has been constructed as the union of $4^n$ discs of radius $r_n$. We will construct $K_{n+1}$ as follows: For each disc $D \subset K_n$, place four discs inside $D$ of radius $r_{n+1}$, each internally tangent to $\partial D$, with their centers forming the vertices of a square rotated by some angle $\omega_D$. (See \Cref{fig:rotated square}.) Let $K_{n+1}$ be the union of these $4^{n+1}$ discs.
    \item 
    The final set is $K = \bigcap_{n=0}^\infty K_n$.
\end{enumerate} 

Then $K$ is completely determined by the choice of radii $r_0 > r_1 > \cdots$ and the choice of rotation angles $\omega_D$ -- one for each disc at every stage of the construction. (This is stated more precisely in \Cref{section:definition}.)

\begin{figure}[h]
\centering
\includegraphics[page=2,width=0.23\textwidth]{standard_four_corner1.pdf}
\qquad
\includegraphics[page=2,width=0.23\textwidth]{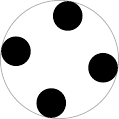}
\qquad
\includegraphics[page=6,width=0.23\textwidth]{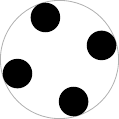}
\caption{The outer disc represents $D \subset K_n$. The four smaller discs represent the four discs we place inside $D$. The left is the usual axis-aligned square. The other two are rotated by some angle.}
\label{fig:rotated square}
\end{figure}

Here are two natural questions about the Favard length of sets obtained in this manner.

\begin{problem}\label{prob:qual}
    What are the sequences $(r_n)_{n\ge 0}$ and rotations $(\omega_D)$ such that the corresponding Cantor set $K$ satisfies $\Fav(K)>0$? 
\end{problem}

It is natural to ask for a more quantitative version of Problem \ref{prob:qual}.
\begin{problem}\label{prob:quant}
    Prove estimates for $\Fav(K_n)$ in terms of $(r_k)_{k=0}^n$ and $(\omega_D)$.
\end{problem}

In the special case where all the rotation angles $\omega_D$ are equal to 0 and $r_{k+1}<r_k/2$, \Cref{prob:qual} is a well-known question that already appeared, e.g., in \cite[Problem 7]{mattila2004hausdorff} and \cite[\textsection 6.6]{volberg-eiderman}. In this situation, the sets $K$ are the well-known non-homogeneous Cantor sets introduced in the 70s by Garnett \cite{garnett1974analytic}. The standard construction uses squares of variable sidelengths instead of discs, but both constructions are equivalent. Motivated by Vitushkin's conjecture (see \Cref{subsec:Vit}), it would be very desirable to characterize sequences $(r_n)_{n\ge 0}$ for which $\Fav(K)>0.$ 

It is not difficult to show that $\cH^1(K)<\infty$ if and only if $\sup_{k\ge 0} 4^kr_k<\infty.$ Consequently, for such sequences the Besicovitch projection theorem ensures $\Fav(K)=0$. It was shown in \cite{ps} that this is still the case for sequences such that $4^kr_k$ converges to infinity at an explicit, but very slow rate.

In the other direction, it follows from Mattila's bound \eqref{eq:mat-low-bd} and the estimates for $\Cap_1(K)$ (see Lemma \ref{lem:cap-lower-bound}) that
\begin{equation}\label{eq:low-nonhom}
    \Fav(K_n)\gtrsim \left(\sum_{j=0}^n \frac{1}{4^j r_j}\right)^{-1},
\end{equation}
so certainly $\Fav(K)>0$ provided that $\sum_{j=0}^\infty 4^{-j} r_j^{-1}<\infty$. For \Cref{prob:quant}, we are not aware of any estimates on $\Fav(K_n)$ other than the lower bound \eqref{eq:low-nonhom}. 

\subsection{Random Cantor sets}

In this article we solve a random variant of Problems \ref{prob:qual} and \ref{prob:quant}. We make the rotations $\omega_D$ random, so that $K$ is a random set. We consider two kinds of random models of this kind, which we refer to as \emph{synchronized} and \emph{independent}:

\begin{itemize}
    \item \textbf{Synchronized rotations model}: Let $\omega_0, \omega_1, \omega_2, \ldots$ be IID uniform random variables in $[0,2\pi]$. We apply the same rotation $\omega_n$ to every disc $D \subset K_n$ in the $n$th generation. See \Cref{fig:synchronized}.
    \item \textbf{Independent rotations model}: The random rotation is uniform and independent across \emph{all} discs. See \Cref{fig:independent}.
\end{itemize}

\begin{figure}[h]
\centering
\includegraphics[page=5,width=0.23\textwidth]{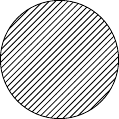}
\hfill
\includegraphics[page=6,width=0.23\textwidth]{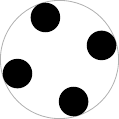}
\hfill
\includegraphics[page=7,width=0.23\textwidth]{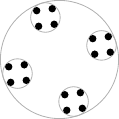}
\hfill
\includegraphics[page=8,width=0.23\textwidth]{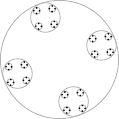}
\caption{Synchronized rotations.}
\label{fig:synchronized}
\end{figure}

\begin{figure}[h]
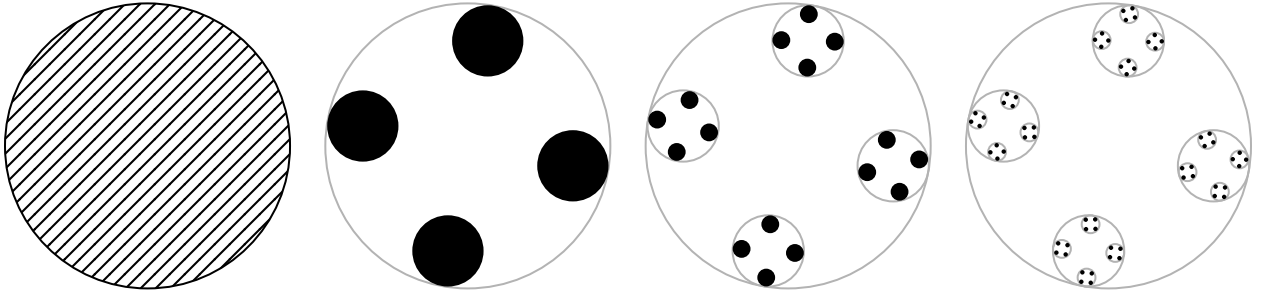

\centering
\includegraphics[page=1,width=0.23\textwidth]{random_discs1.pdf}
\hfill
\includegraphics[page=2,width=0.23\textwidth]{random_discs1.pdf}
\hfill
\includegraphics[page=3,width=0.23\textwidth]{random_discs1.pdf}
\hfill
\includegraphics[page=4,width=0.23\textwidth]{random_discs1.pdf}
\caption{Independent rotations.}
\label{fig:independent}
\end{figure}

In the homogenous case $r_n = 4^{-n}$, the lower bound \eqref{eq:low-nonhom} implies that for any choice of rotations (no randomness required), the Cantor sets obtained this way satisfy
\[
\Fav(K_n) \gtrsim \frac{1}{n}.
\]
For $r_n = 4^{-n}$, Zhang \cite{zhang2020} studied the synchronized rotations model and Vardakis--Volberg \cite{vardakis-volberg} studied the independent rotations model, and they showed that Mattila's lower bound is sharp in expectation. That is, for both models, we have
\[
\E \Fav(K_n) \lesssim \frac{1}{n}.
\]

\subsection{New results}
Our main result solves a random variant of Problem \ref{prob:quant}.

\begin{theorem}[Sharp expected Favard length decay]\label{thm:main}
Fix a sequence $r_0 \geq r_1 \geq r_2 \geq \cdots$ satisfying $r_{n+1} \leq \frac{1}{3} r_n$ for all $n$. Then for both the synchronized and independent models, we have:
\begin{align}
\label{eq:E Fav Kn main thm}    
    \E[\Fav(K_n)]
    &\approx
    \left(
    \sum_{j=0}^n \frac{1}{4^j r_j}
    \right)^{-1}
    &&\text{for all $n$.}
\\
\label{eq:E Fav K(r) main thm}
    \E[\Fav(K(r))]
    &\approx
    \left(
    \sum_{j=0}^{n-1} \frac{1}{4^j r_j}
    +
    \frac{1}{4^n r}
    \right)^{-1}
    &&\text{for all $n$ and $r \in [r_n, r_{n-1}]$}.
\end{align}
\end{theorem}

\begin{remark}
\label{definition:admissible}
We say that a sequence of radii $r_0 \geq r_1 \geq r_2 \geq \cdots$ is \emph{admissible} if 
\begin{align}
\label{eq:admissible}    
r_{n+1} \leq \frac{1}{3} r_n \qquad\text{for all $n$}.
\end{align}
Throughout this paper, we adopt the convention $r_{-1} = \infty$, so that \Cref{thm:main} also applies for $r \geq r_0$. We also allow the degenerate case where $r_{n-1} > 0 = r_{n} = r_{n+1} = \cdots$ for some $n$. In that case, $K_n$ is a finite set of $4^n$ points, and $K = K_n = K_{n+1} = \cdots$.

The constant $\frac{1}{3}$ in \eqref{eq:admissible}  is arbitrary; we could have also chosen any number strictly between $\frac{1}{4}$ and $\frac{1}{1+\sqrt{2}}$. If $\frac{r_{n+1}}{r_n} \geq \frac{1}{1+\sqrt{2}}$, the discs in $K_{n+1}$ will overlap.
\end{remark}

The lower bounds in \Cref{thm:main} are proved by combining Mattila's lower bound \eqref{eq:mat-low-bd} with the following capacity estimates. Note that this statement does not require any randomness.

\begin{lemma}\label{lem:cap-lower-bound}
For any admissible sequence of radii $(r_n)_{n\ge 0}$ and any choice of rotations, the associated Cantor set $K$ satisfies
\begin{align}
\label{eq:Cap K_n lower bd}  
    \Cap_1(K_n)
    &\gtrsim
    \left(
    \sum_{j=0}^n \frac{1}{4^j r_j}
    \right)^{-1}
    &&\text{for all $n$.}
\\
\label{eq:Cap K(r) lower bd}
    \Cap_1(K(r))
    &\gtrsim
    \left(
    \sum_{j=0}^{n-1} \frac{1}{4^j r_j}
    +
    \frac{1}{4^n r}
    \right)^{-1}
    &&\text{for all $n$ and $r \in [r_n, r_{n-1}]$}.
\end{align}
\end{lemma}

We prove \Cref{lem:cap-lower-bound} (and hence the lower bounds in \Cref{thm:main}) in \Cref{sec:lower}. These estimates were known in the case of non-homogeneous Cantor sets with no rotations, and the proof is virtually the same. We could not find it written down anywhere, so we include the full proof for completeness. 

Next, we prove the upper bounds in \Cref{thm:main} in \Cref{sec:upper-bd}. The idea of the proof is similar to \cite{zhang2020} and \cite{vardakis-volberg} and involves identifying a recurrence relation between various stages of the construction. Our recurrence is more intricate.

As a corollary of \Cref{thm:main} and \Cref{lem:cap-lower-bound}, we get a characterization of sequences $(r_n)_{n\ge 0}$ such that the corresponding random Cantor sets almost surely have positive Favard length, answering a random version of Problem \ref{prob:qual}.
\begin{corollary}\label{cor:qualit}
    Fix an admissible sequence $(r_n)_{n\ge 0}$. If $\sum_{j=0}^\infty 4^{-j} r_j^{-1}<\infty$, then for any choice of rotations the associated Cantor set satisfies $\Fav(K)>0.$ 
    
    Conversely, if $\sum_{j=0}^\infty 4^{-j} r_j^{-1}=\infty$, and we consider either the independent or synchronized rotations model, then almost surely the associated Cantor set satisfies $\Fav(K)=0.$
\end{corollary}
\begin{proof}
    The first statement follows from the deterministic lower bound in Lemma \ref{lem:cap-lower-bound} together with \eqref{eq:mat-low-bd}. For the second statement, suppose that $\sum_{j=0}^\infty 4^{-j} r_j^{-1}=\infty$. Then, by letting $n\to\infty$ in \eqref{eq:E Fav Kn main thm}, we get $\E\Fav(K)=0$. Since $\Fav(K)$ is a non-negative random variable, this implies $\Fav(K)=0$ almost surely.
\end{proof}

The prior work on random disc models by \cite{zhang2020} and \cite{vardakis-volberg} only gave results about expectation. By following the arguments of \cite{chang-shmerkin-suomala}, we show that for the independent rotations model, the expectation bounds can be converted to almost sure bounds.

\begin{theorem}[Converting expected decay to almost sure decay]
\label{theorem:expected to almost sure decay}
Consider the independent rotations model. Fix an admissible sequence $r_0 \geq r_1 \geq r_2 \geq \cdots$, and suppose $\lim_{r \to 0} \E \Fav(K(r)) = 0$. Then 
\begin{align}\label{eq:almostsure}
    \limsup_{r \to 0} \frac{\Fav K(r)}{\E \Fav K(r)} \leq C
    \qquad\text{almost surely}.
\end{align}
The constant $C$ is absolute.
\end{theorem}

We prove \Cref{theorem:expected to almost sure decay} in \Cref{section:expected to almost sure decay}. 

The assumption $\lim_{r \to 0} \E \Fav(K(r)) = 0$, or equivalently $\sum_{j=0}^\infty 4^{-j} r_j^{-1}=\infty$, is necessary in Theorem \ref{theorem:expected to almost sure decay}; see \Cref{remark:lim=0 necessary}.

\begin{remark}
All of our results generalize in a straightforward manner if the construction places $L \geq 3$ sub-discs at each stage instead of $4$. In this setting, we would need to change the constant in the definition of admissibility \eqref{eq:admissible} to ensure that sub-discs at each step are separated. The implied constants will depend on $L$. We present the statements and proofs for the $L=4$ case to avoid additional notation.
\end{remark}

\subsubsection{Prescribed Favard length decay} 

The previous theorems started with a fixed sequence $(r_n)_{n \geq 0}$ and determined the decay rate of $\E \Fav(K(r))$. One can ask the reverse question: given a prescribed decay rate for $\E \Fav(K(r))$, is there a sequence $(r_n)$ that realizes it? This amounts to analyzing the right-hand side of \eqref{eq:E Fav K(r) main thm}. We give a complete characterization, up to absolute constants, of the decay rates that can be attained in this way.

\begin{theorem}[Prescribed expected Favard length decay]
\label{theorem:prescribed}
    Let $f:(0,\infty)\to (0,\infty)$ be a concave increasing function satisfying $\lim_{r \to \infty} f'(r) = 1$. Then there exists an admissible sequence $(r_n)_{n\ge 0}$ such that for both the synchronized and independent models,
    \begin{equation*}
        \E \Fav(K(r))
        \approx  
        f(r) \quad\text{for all $r > 0$}.
    \end{equation*}
    The implied constant is absolute.
\end{theorem}

We prove \Cref{theorem:prescribed} in \Cref{section:prescribe}.

\begin{remark}
    This is sharp in the sense that there are no other functions $f$ to consider: for any compact set $E \subset \R^n$, 
    \begin{align}
        \label{eq:f conc incr}
    f(r) := \Fav(E(r)) \text{ is a concave increasing function satisfying } \lim_{r \to \infty} f'(r) = 2\pi.
    \end{align}
    
    To see this, first consider a compact set $K \subset \R$. For any $r > 0$, let $N(r)$ be the number of connected components of the open $r$-neighborhood $K(r)$. Then $N(r)$ is a right-continuous, non-increasing function such that $N(r) = 1$ for all sufficiently large $r$. Let $f_K(r) = |K(r)|$. Note that if $N(r) = N_0$ on some interval $(a,b)$, then $f_K'(r) = 2N_0$ on $(a,b)$. This shows that $f_K(r)$ is a concave increasing function with $f_K'(r) = 2$ for all sufficiently large $r$. 

    Furthermore, this is a complete characterization in $\R$. See \Cref{section:prescribe countable} for details.

    For any compact $E \subset \R^2$, the argument above shows that for any $\theta$, the function $r \mapsto |\proj_\theta(E(r))|$ is concave increasing and derivative equal to $2$ for all $r \geq \operatorname{diam}(E)$. Integrating over $\theta \in [0,\pi]$ gives us \eqref{eq:f conc incr}.
\end{remark}

Combining the results above gives the following. 

\begin{corollary}[Existence of Cantor sets for any prescribed decay rate]
\label{corollary:prescribed}
    Let $f:(0,\infty)\to (0,\infty)$ be a concave increasing function satisfying $\lim_{r \to 0} f(r) = 0$ and $\lim_{r \to \infty} f'(r) = 1$. Then there exists an admissible sequence $r_0 \geq r_1 \geq r_2 \geq \cdots$ such that for the independent rotations model, the following holds almost surely:
    \begin{equation*}
    	\text{There exists $\epsilon > 0$ such that for all $0 < r < \epsilon$,}\qquad
        \Fav(K(r))\approx \Cap_1(K(r)) \approx f(r).
    \end{equation*}
    The implied constants are absolute.
\end{corollary}

One consequence of \Cref{corollary:prescribed} is the sharpness of Mattila's lower bound \eqref{eq:mat-low-bd}. Until recently, it was not known whether there existed any set $E \subset \R^2$ with $\Fav(E) = 0$ for which Mattila's lower bound $\Fav(E(r)) \gtrsim \Cap_1(E(r))$ is sharp for all $r > 0$. The recent work \cite{chang-shmerkin-suomala} gave an affirmative answer, but the examples there were restricted to the logarithmic decay rate $\Cap_1(E(r)) \approx 1/\log(1/r)$. In contrast, \Cref{corollary:prescribed} shows that Mattila's bound is sharp for any decay rate.

\begin{remark}
If the goal is to show that for any such $f$, we can find a set $K \subset \R^2$ with the desired Favard decay rate, that can be easily done by considering sets $K$ which are contained inside a line; see \Cref{section:prescribe countable}. However, such sets would have $1$-capacity zero.
\end{remark}

\begin{remark}
In \cite{ps2005}, Peres and Solomyak use random non-homogeneous grid-based Cantor sets to study the Hausdorff measure $\cH^\phi$ associated with a gauge function $\phi$. For any non-decreasing function $\phi : [0,\infty) \to [0, \infty)$ such that $\phi(r)/r^2$ is non-increasing and $\int_0^1 \frac{\phi(r)}{r^2} \, dr = \infty$, they construct a model of random Cantor sets $K \subset \R^2$ such that $\cH^\phi(K) > 0$ deterministically and $\Fav(K) = 0$ almost surely; see \cite[Proposition 2.1]{ps2005}. While one can infer upper bounds on $\E \Fav(K_n)$ from their arguments (see \cite[Lemma 4.1]{ps2005}), quantitative estimates were not the focus of their paper.
\end{remark}

\subsection{Vitushkin's conjecture}\label{subsec:Vit}

We say that a compact set $E \subset \C$ is removable for bounded analytic functions if for any open set $\Omega \supset E$, every bounded analytic function on $\Omega\setminus E$ has an analytic extension to $\Omega$. Ahlfors \cite{ahlfors} introduced the analytic capacity of a set $E$, defined by
\[
\gamma(E) = \sup |f'(\infty)|,
\]
where the supremum is over all analytic functions $f : \C \setminus E \to \C$ with $|f| \leq 1$ and $f'(\infty) = \lim_{z \to \infty} z(f(z) - f(\infty))$. Ahlfors showed that $E$ is removable for bounded analytic functions if and only if $\gamma(E) = 0$.

In the 1960s, Vitushkin conjectured that $\gamma(E) > 0$ if and only if $\Fav(E) > 0$. Mattila \cite{mattila1986} showed that Vitushkin's conjecture is false. Jones and Murai \cite{jones-murai}, and later Joyce and M\"orters \cite{joyce-morters} gave examples of sets satisfying $\Fav(E) = 0$ and $\gamma(E) > 0$. We prove that our random construction provides a new class of examples which are simpler than those from \cite{jones-murai,joyce-morters}. Let $\theta_k\coloneqq 4^{-k}/r_k.$

\begin{theorem}\label{thm:anal-cap}
    For any admissible sequence of radii $(r_n)_{n\ge 0}$ and any choice of rotations, the associated Cantor set $K$ satisfies
    \begin{equation}\label{eq:analytic-cap}
        \gamma(K_n)\approx \left(\sum_{k=0}^{n}\theta_k^2\right)^{-1/2}=\left( \sum_{k=0}^{n} \frac{1}{4^{2k} r_k^2} \right)^{-1/2}.
    \end{equation}
\end{theorem}
Letting $n\to\infty$ and using the outer regularity of Favard length and analytic capacity (see \cite[Proposition 1.7]{tolsa2014analytic}) we get that $\gamma(K)>0$ if and only if $\sum_{k=0}^\infty \theta_k^2 < \infty$. Together with Corollary \ref{cor:qualit} this gives the following.
\begin{corollary}
    Consider an admissible sequence of radii $(r_n)_{n\ge 0}$ satisfying $\sum_{k=0}^\infty \theta_k = \infty$ and $\sum_{k=0}^\infty \theta_k^2 < \infty$ (e.g., $r_k=k 4^{-k}$). Let $K$ be a corresponding random Cantor set obtained with either the independent or synchronized rotations model. Then $\gamma(K) > 0$, and almost surely $\Fav(K) = 0$.
\end{corollary}

Originally the problem of estimating $\gamma(K)$ for non-homogeneous Cantor sets of this type (without rotations) was considered by Garnett \cite{garnett1974analytic}. The sharp lower bound in this case is due to Mattila \cite{mattila1996}, and the matching upper bound is due to Eiderman \cite{eiderman1998} and Mateu, Tolsa, and Verdera \cite{mtv}. Adding the rotations changes little, but for completeness we include the proof in Section \ref{sec:analytic-capacity}. We followed Section 4.7 in \cite{tolsa2014analytic}.

We mentioned above that the implication $\gamma(E)>0 \implies \Fav(E)>0$ is false in general (for sets of finite length it was shown to be true by David \cite{david1998unrectifiable}). Whether the converse implication 
\begin{equation}\label{eq:vit}
    \Fav(E)>0\quad\implies\quad \gamma(E)>0
\end{equation}
holds remains an open problem -- it is only known to be true for sets of $\sigma$-finite length by \cite{calderon1977cauchy}. It follows from Corollary \ref{cor:qualit} and \eqref{eq:analytic-cap} that for the random Cantor sets we consider the implication \eqref{eq:vit} is true almost surely. 

The implication \eqref{eq:vit} would be implied by the following estimate, conjectured e.g. in \cite{ivanov1994onsets}:
\begin{equation}\label{eq:quan-Vit}
    \gamma(E)\gtrsim\Fav(E).
\end{equation}
It follows from Theorems \ref{thm:main} and \ref{thm:anal-cap} that \eqref{eq:quan-Vit} holds in expectation for our random Cantor sets.
\begin{corollary}
     For any admissible sequence of radii $(r_n)_{n\ge 0}$, in both the synchronized and independent random models, the associated Cantor sets satisfy $\E\gamma(K)\gtrsim\E\Fav(K).$
\end{corollary}
For recent articles making partial progress on \eqref{eq:quan-Vit}, see \cite{ct} and \cite{dabrowski-villa2025,dabrowski2024}. For more information on the Vitushkin's conjecture and analytic capacity, see Tolsa's monograph \cite{tolsa2014analytic}.

\subsection{Acknowledgments}

AC was partially supported by NSF grant DMS-2247233. DD was supported by the European Union (ERC 101219218 QPROJECT). Views and opinions expressed are however those of the authors only and do not necessarily reflect those of the European Union or the European Research Council. Neither the European Union nor the granting authority can be held responsible for them.

This project started at the \emph{GMTPyrenees} workshop in June 2025 in Sant Pau de Seg\'uries, Catalonia. The workshop was funded by the European Union MSCA Individual Fellowship QuReViMal, no.\ 101108515. The authors thank the workshop organizers for the idyllic work environment, and Giovanni Alberti for helpful comments. LLMs were used to generate some of the figures and for proofreading.

\section{Definitions and basic properties}
\label{section:definition}

\subsection{Discs}

Throughout this paper, all discs will be assumed to be closed and nondegenerate. For $x \in \R^2$ and $r > 0$, let $B(x,r)$ be the closed disc centered at $x$ of radius $r$.

For $r \in (0,1)$ and $\omega \in [0,2\pi]$, we define $P(r,\omega) \subset B(0,1)$ to be the ``peripheral'' sub-disc of radius $r$ that is internally tangent to $\partial B(0,1)$ at $e^{i\omega}$, i.e.,
\begin{align}
P(r,\omega) = B((1-r)e^{i\omega},r).
\end{align}

For any disc $D \subset \R^2$, let $\discmap{D} : \R^2 \to \R^2$ be the unique map of the form $z \mapsto rz + v$ ($r>0, v \in \R^2$) such that $\discmap{D}$ maps $B(0,1)$ onto $D$.  We define a binary operation on the set of all discs as follows:
\begin{align}
\discmap{D_1 \circ D_2} = \discmap{D_1} \circ \discmap{D_2}
\qquad
\text{or equivalently,}
\qquad
D_1 \circ D_2 = \discmap{D_1}(D_2)
.
\end{align}
See \Cref{figure:disc operation}.

\begin{figure}[h]
\centering
\includegraphics[page=1,width=0.35\textwidth]{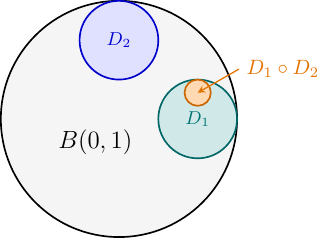}
\caption{In this example, $D_1 = P(1/3,0)$ and $D_2 = P(1/3,\pi/2)$}
\label{figure:disc operation}
\end{figure}

\begin{remark}
The binary operation defined above is in fact a group operation, and the group is isomorphic to the matrix group
\[
\left\{
\begin{pmatrix}
r & 0 & x 
\\
0 & r & y
\\
0 & 0 & 1
\end{pmatrix}
:
r > 0, x \in \R, y \in \R
\right\}.
\]
\end{remark}

\subsection{The Cantor set}

Let $[4] = \{0,1,2,3\}$. Let $[4]^* = \bigsqcup_{k=0}^\infty [4]^k$ denote the set of all finite sequences in $[4]$. The elements of $[4]^*$ can be arranged naturally in a tree structure with the empty sequence $\varnothing$ as the root. For a sequence $\ii = i_1 i_2 \cdots i_k \in [4]^k$, we denote its length by $|\ii|$ and define it to be $k$, and we denote its children by $\ii j$, $j \in [4]$. If $k \geq 1$, we denote the parent of $\ii$ by $\ii'$ and define it to be $i_1 i_2 \cdots i_{k-1} \in [4]^{k-1}$.

Given an admissible sequence of radii $(r_n)_{n\ge 0}$, and a sequence of angles $(\omega_\ii)_{\ii \in [4]^*}$ in $[0,2\pi]$, we define a disc-like Cantor set $K = \bigcap_{n=0}^\infty K_n$ as follows.

First, for each $\ii \in [4]^*$ and $j \in [4]$, we define the peripheral disc
\begin{align}
\label{eq:def P_ij}
P_{\ii j} = P\left(\frac{r_{|\ii|+1}}{r_{|\ii|}}, \omega_{\ii} + \frac{2\pi j}{4}\right)
.
\end{align}
Then we define a tree of discs $(D_\ii)_{\ii \in [4]^*}$ recursively by
\begin{align*}
D_\varnothing &= B(0,r_0)
\\
D_{\ii j}
&= 
D_\ii 
\circ 
P_{\ii j}
.
\end{align*} 
Equivalently, for all $\ii \in [4]^*$,
\begin{align}
\label{eq:def D_i non-recursive}
D_\ii
=
B(0,r_0) 
\circ 
P_{i_1}
\circ
P_{i_1i_2}
\circ
\cdots
\circ 
P_\ii
.
\end{align}
See \Cref{figure:def D_i}.

\begin{figure}[h]
\centering
\includegraphics[page=1,width=0.35\textwidth]{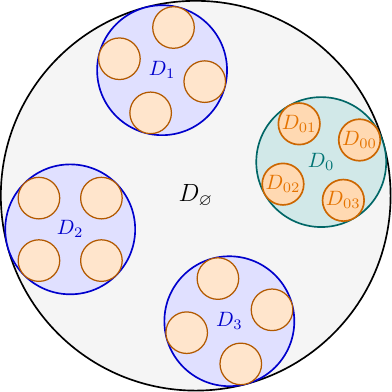}
\qquad
\includegraphics[page=2,width=0.35\textwidth]{d_i1.pdf}
\caption{Illustration of $D_{0 j} = D_0 \circ P_{0 j}$}
\label{figure:def D_i}
\end{figure}

Observe the following properties:
\begin{enumerate}
\item 
\textbf{Nestedness:} $D_\ii \subset D_{\ii'}$.
\item 
\textbf{Radii:} The radius of $D_\ii$ is $r_{|\ii|}$.
\item 
\textbf{Dependence on parameters:} The set of discs $\{D_\ii\}_{|\ii| \leq k}$ up to generation $k$ only depends on the values  $(\omega_\ii)_{|\ii| < k}$ and $(r_j)_{j \leq k}$. 
\item 
\textbf{Separation:} For any distinct $\ii,\ii'\in[4]^k$, we have
\begin{equation}\label{eq:separation}
    \dist(D_{\ii}, D_{\ii'})\ge \frac{r_{k-1}}5.
\end{equation}
This follows from admissibility (and is in fact the only way that admissibility is used).
\end{enumerate}

We define $\cD_n = \{D_\ii : |\ii| = n\}$ and $K_n = \cup \cD_n$. Then the nestedness of the discs implies $K_0 \supset K_1 \supset K_2 \supset \cdots$, and we define our Cantor set to be $K = \bigcap_{n=0}^\infty K_n$.

\subsection{The random models}
\label{section:random models}

In our random models, the sequence of radii $(r_k)_k$ is fixed, and the randomness comes only from the angles $\omega_\ii$. We consider two models:

\begin{itemize}
    \item \textbf{Synchronized rotations model}: Let $\omega_0, \omega_1, \omega_2, \ldots$ be IID uniform random variables in $[0,2\pi]$. Then we let $\omega_\ii = \omega_{|\ii|}$.
    \item \textbf{Independent rotations model}: All the angles $(\omega_\ii)_{\ii \in [4]^*}$ are independent uniform random variables in $[0,2\pi]$.
\end{itemize}

\subsection{$K_n$ vs $K(r)$}
\label{section:K_n vs K(r)}

In this section, we record some connections between the sets $K_n$ and $K(r)$. No randomness is used. Recall that $K(r)$ denotes the closed $r$-neighborhood of $K$.

One heuristic, which can already seen in \Cref{thm:main}, is that $K(r)$ is ``roughly the same'' as $K_n$ when $r = r_n$. For intermediate scales $r \in [r_n,r_{n-1}]$, some ``interpolation'' is needed. This is done with an auxiliary Cantor set $\wt K_0 \supset \wt K_1 \supset \cdots \supset \wt K_n$ defined via an auxiliary admissible sequence of radii $\wt r_0 \geq \wt r_1 \geq \cdots \geq \wt r_n$. We make these ideas precise below.

\begin{lemma}[Upper bound on $K(r)$]
\label{lemma:K(r) upper bound K_n}
Fix an admissible sequence $r_0 \geq r_1 \geq r_2 \geq \cdots$, and a sequence of angles $(\omega_\ii)_{\ii \in [4]^*}$. Fix $r > 0$. Let $n$ be such that $r \in [r_n,r_{n-1}]$. We let $\wt K_n$ be defined using the angles $(\omega_\ii)_{\ii}$ from above, but with the modified admissible radii $(\wt r_0, \ldots, \wt r_n)$, where
\[
\wt r_0 = r_0, \qquad \ldots, \qquad \wt r_{n-1} = r_{n-1},
\qquad 
\wt r_n
=
\min\left(r, \frac{1}{3}r_{n-1}\right).
\]
Then $\Fav K(r) \lesssim \Fav \wt K_n$
\end{lemma}

\begin{proof}
Since $\wt r_n \geq r_n$, we have $K \subset K_n \subset \wt K_n$, so
\begin{align}
\label{eq:Fav K(r) first step}
    K(r) \subset \wt K_n(r).
\end{align}

For fixed $\theta$, $\proj_{\theta} \wt K_n$ is a union of disjoint intervals of length $\ge \wt r_n \ge \frac{1}{3}r$. It follows that there are $\le (\frac{1}{3}r)^{-1}|\proj_{\theta} \wt K_n|$ such intervals. Observe that for any set $A\subset\R$ which is a union of $N$ intervals and any $\delta>0$, we have $|A(\delta)\setminus A|\le 2 N \delta$. Applying this observation with $A = \proj_\theta \wt K_n$ and $\delta = r$ gives 
$
|(\proj_\theta \wt K_n)(r)|
\lesssim 
|\proj_\theta \wt K_n|
.
$
Integrating in $\theta$,
\begin{align}
\label{eq:Fav tilde K n+1 r ineq}
    \Fav \wt K_n(r)
    \lesssim
    \Fav \wt K_n
\end{align}
The lemma follows by combining \eqref{eq:Fav K(r) first step} and \eqref{eq:Fav tilde K n+1 r ineq},
\end{proof}

\begin{lemma}[Lower bound on $K(r)$]
    \label{lemma:K_n in K(r)}
    Fix an admissible sequence $r_0 \geq r_1 \geq r_2 \geq \cdots$, and a sequence of angles $(\omega_\ii)_{\ii \in [4]^*}$. Fix $r > 0$. Let $n$ be such that $r \in [2r_n, 2r_{n-1}]$. We let $\wt K_n$ be defined using the angles $(\omega_\ii)_{\ii}$ from above, but with the modified admissible radii $(\wt r_0, \ldots, \wt r_n)$, where
    \[
    \wt r_0 = r_0, \qquad \ldots, \qquad \wt r_{n-1} = r_{n-1},
    \qquad 
    \wt r_n
    =
    \min\left(\frac{1}{2}r, \frac{1}{3}r_{n-1}\right).
    \]
    Then $\wt K_n \subset K(r)$.
\end{lemma}

\begin{proof}
Since $\wt K_n$ is the union of discs $\{\wt D_\ii\}_{\ii \in [4]^n}$, it suffices to show $\wt D_\ii \subset \overline{K(r)}$ for all $\ii \in [4]^n$. Fix $\ii \in [4]^n$. Since  $r \geq 2r_n$ and the sequence $(r_k)$ is admissible, we have
\[
\text{radius of $D_\ii$} = r_n \leq \wt r_n = \text{radius of $\wt D_\ii$}.
\]
Thus, $D_\ii \subset \wt D_\ii$. (Here, we also use the facts that $K_{n-1} = \wt K_{n-1}$ and that the same angles $(\omega_\jj)$ are used to define both $D_\ii$ and $\wt D_\ii$.) By construction $D_\ii \cap K \neq \emptyset$, so let $x \in D_\ii \cap K$. Then $x \in D_\ii \subset \wt D_\ii$, so $\wt D_\ii \subset B(x,r) \subset \overline{K(r)}$, as desired.
\end{proof}

\section{The lower bound in \Cref{thm:main} via \Cref{lem:cap-lower-bound}}\label{sec:lower}

In this section, we prove \Cref{lem:cap-lower-bound}. Let $\mu_n$ be the natural probability measure on $K_n$:
\begin{equation}
\label{eq:def m_n}
    \mu_n\coloneqq \frac{1}{\mathcal{L}^2(K_n)}\mathcal{L}^2|_{K_n} = \frac{1}{4^n \pi r_n^2}\mathcal{L}^2|_{K_n},
\end{equation}
so that $\mu_n(D_{\ii})=4^{-k}$ for any $\ii\in [4]^k$, $0\le k\le n$. By the definition of $\Cap_1$ we have
\begin{equation*}
    \Cap_1(K_n)\ge \frac{1}{I_1(\mu_n)},
\end{equation*}
so it suffices to estimate $I_1(\mu_n)$ from above. To do this, we will bound the densities of $\mu_n$. Given a measure $\nu$, $x\in\spt \nu$, and $r>0$, the $1$-dimensional density of $\nu$ is defined as
\begin{equation*}
    \theta_\nu(x,r)\coloneqq \frac{\nu(B(x,r))}{r}.
\end{equation*}
Note that
\begin{multline}\label{eq:ener-density}
    I_1(\nu)=\iint \frac{d\nu(x)d\nu(y)}{|x-y|}  = \iiint_{|x-y|}^\infty \frac{dr}{r^2}\, d\nu(x)d\nu(y)\\
    =\int_0^\infty \int_{x\in\R^2}\int_{y\in B(x,r)} r^{-1} d\nu(y)d\nu(x)\frac{dr}{r}=\iint_0^\infty\theta_\nu(x,r)\, \frac{dr}{r}d\nu(x).
\end{multline}
Thus, estimates on $\theta_{\mu_n}(x,r)$ imply estimates on $I_1(\mu_n)$.

\begin{lemma}[Density estimates]\label{lem:dens-est}
    For any $k\in \{0, \dots, n\}$ we have
    \begin{align}
        \label{eq:densities}
        \theta_{\mu_n}(x,r)\approx \frac{1}{4^{k}r} \qquad \text{for all $x\in K_n$ and $r \in [r_{k}, r_{k-1}]$}
    \end{align}
    (For $k=0$, recall the convention $r_{-1} = \infty$; see \Cref{definition:admissible}.) Also,
    \begin{align}
        \label{eq:densities small r}
        \theta_{\mu_n}(x,r)\approx \frac{r}{4^{n}r_n^2} \qquad \text{for all $x\in K_n$ and $r \in (0, r_{n}]$}
    \end{align}
\end{lemma}

\begin{proof}
    We start by proving \eqref{eq:densities small r}. Since $K_n$ is a union of discs of radius $r_n$, it follows that for any $x \in K_n$ and $r \in (0, r_n]$, the intersection $B(x,r) \cap K_n$ contains a disc of radius $r/2$. Thus, we have
    \begin{equation}\label{eq:measure lower bound small r}
        \mu_n(B(x,r)) 
        \stackrel{\eqref{eq:def m_n}}{=}
        \frac{\cL^2(B(x,r) \cap K_n)}{4^n \pi r_n^2}
        \geq \frac{1}{4} \frac{r^2}{4^nr_n^2}
        \qquad\text{for all $x \in K_n$ and $r \in (0,r_n]$}.
    \end{equation}
    If we combine this with the trivial upper bound $\cL^2(B(x,r) \cap K_n) \leq \cL^2(B(x,r)) = \pi r^2$, then we obtain \eqref{eq:densities small r}.

    It remains to prove \eqref{eq:densities}. To do this, it suffices to show
    \begin{equation}\label{eq:dens1}
        4^{-k-1}\le\mu_n(B(x,r))\le 4^{-k+2}\qquad \text{for all $x\in K_n$ and $r \in [r_{k}, r_{k-1}]$}.
    \end{equation}
    
    We start with the upper bound in \eqref{eq:dens1}.  For $k \in \{0,1,2\}$, the estimate is trivial because $\mu_n$ is a probability measure. So suppose $k \in \{3, \ldots, n\}$ and fix $x\in K_n$ and $r \in [r_{k}, r_{k-1}]$. Let $D_\ii$ be the unique disc with $|\ii|=k-2$ containing $x$. For any other disc $D_{\ii'}$ with $|\ii'|=k-2$, we have
    \begin{equation*}
        \dist(D_\ii, D_{\ii'})
        \stackrel{ \eqref{eq:separation}}{\ge} \frac{r_{k-3}}5\ge  \frac{3^2 r_{k-1}}{5}>r_{k-1} \geq r,
    \end{equation*}
    which shows that $B(x,r) \cap D_{\ii'} = \emptyset$.
    Consequently,
    \begin{equation*}
        \mu_n(B(x,r)) = \mu_n(B(x,r)\cap D_\ii)\le\mu_n(D_\ii)=4^{-k+2}.
    \end{equation*}

    We move to the lower bound in \eqref{eq:dens1}. For $k=n$, we have, for $x \in K_n$ and $r \in [r_n, r_{n-1}]$,
    \[
    \mu_n(B(x,r))
    \geq
    \mu_n(B(x,r_n))
    \stackrel{\eqref{eq:measure lower bound small r}}{\geq} 
    4^{-n-1}.
    \]
    Suppose $k \in \{0, \ldots, n-1\}$. Let $x \in K_n, r\in [r_k,r_{k-1}]$, and let $D_\ii$ be the disc containing $x$ with $|\ii|=k+1$. Since $r\ge r_{k}\ge 3r_{k+1}$ we have $D_{\ii}\subset B(x,2r_{k+1}) \subset B(x,r)$, so that $\mu_n(B(x,r))\ge \mu_n(D_\ii)=4^{-k-1}.$
\end{proof}

\begin{lemma}\label{lem:energy-est}
    We have
    \begin{equation*}
        I_1(\mu_n) \approx \sum_{k=0}^{n} \frac{1}{4^k r_k} 
        .
    \end{equation*}
\end{lemma}
\begin{proof}
    In the light of \eqref{eq:ener-density} and since $\mu_n$ is a probability measure, it is enough to show that 
    \begin{equation}\label{eq:goal1}
        \int_0^\infty\theta_{\mu_n}(x,r)\, \frac{dr}{r} \approx \sum_{k=0}^{n}\frac{1}{4^k r_k}
        \qquad\text{for any $x\in K_n$}.
    \end{equation}
    Fix $x \in K_n$. We decompose
    \begin{equation*}
        \int_0^\infty\theta_{\mu_n}(x,r)\, \frac{dr}{r}
        =
        \sum_{k=0}^n \int_{r_{k}}^{r_{k-1}}\theta_{\mu_n}(x,r)\, \frac{dr}{r}
        +
        \int_0^{r_n}  \theta_{\mu_n}(x,r)\, \frac{dr}{r}
    \end{equation*}
    Using \Cref{lem:dens-est}, we compute
    \begin{equation*}
        \int_{r_{k}}^{r_{k-1}}\theta_{\mu_n}(x,r)\, \frac{dr}{r}
        \stackrel{\eqref{eq:densities}}{\approx} 
        \int_{r_{k}}^{r_{k-1}} \frac{1}{4^{k}r^2}\, dr
        =
        \frac{1}{4^{k}} \left(\frac{1}{r_{k}} - \frac{1}{r_{k-1}}\right)       
        \stackrel{\eqref{eq:admissible}}{\approx}
        \frac{1}{4^{k} r_k}
        .
    \end{equation*}
    and
    \begin{equation*}
        \int_0^{r_n} \theta_{\mu_n}(x,r)\, \frac{dr}{r}
        \stackrel{\eqref{eq:densities small r}}{\approx}  
        \int_0^{r_n} \frac{1}{4^n r_n^2} dr
        =
        \frac{1}{4^n r_n}
        .
    \end{equation*}
    Putting the estimates above together gives \eqref{eq:goal1}.
\end{proof}
\begin{proof}[Proof of Lemma \ref{lem:cap-lower-bound}]
The estimate \eqref{eq:Cap K_n lower bd} follows immediately from Lemma \ref{lem:energy-est}:
\begin{equation*}
    \Cap_1(K_n)\ge I_1(\mu_n)^{-1}\approx \left( \sum_{k=0}^{n} \frac{1}{4^k r_k}\right)^{-1}.
\end{equation*}
Concerning \eqref{eq:Cap K(r) lower bd}, note that if $2r_{n}\le r\le 2r_{n-1}$, then by \Cref{lemma:K_n in K(r)}, $K(r)\supset \wt{K}_n$, where $\wt K_n$ is the auxiliary Cantor set defined using the same rotations $(\omega_\ii)_{\ii}$ as for $K$, but with the admissible radii
$
(\wt r_0, \ldots, \wt r_n)
=
\left(
r_0, \ldots, r_{n-1}, \min\left(\frac{1}{2}r, \frac{1}{3}r_{n-1}\right)
\right).
$
This implies that 
\begin{equation*}
    \Cap_1(K(r))\ge \Cap_1(\wt{K}_n)
    \stackrel{\eqref{eq:Cap K_n lower bd}}{\gtrsim}
    \left( \sum_{k=0}^{n-1} \frac{1}{4^k r_k} + \frac{1}{4^{n} \wt r_{n}}\right)^{-1}\approx \left( \sum_{k=0}^{n-1} \frac{1}{4^k r_k} + \frac{1}{4^{n} r}\right)^{-1}.
\end{equation*}
Thus, we have shown
\begin{equation}
    \Cap_1(K(r)) 
    \gtrsim 
    \left( \sum_{k=0}^{n-1} \frac{1}{4^k r_k} + \frac{1}{4^{n} r}\right)^{-1} \qquad \text{if $r \in [2r_n, 2r_{n-1}]$}.
\end{equation}
which implies \eqref{eq:Cap K(r) lower bd}.
\end{proof}

\section{The upper bound in \Cref{thm:main}}\label{sec:upper-bd}

In this section, we prove the upper bounds of \eqref{eq:E Fav Kn main thm} and \eqref{eq:E Fav K(r) main thm}. Throughout this section, fix $n \in \N$ and admissible sequence $r_0 \geq r_1 \geq r_2 \geq \cdots \geq r_n$.

\subsection{Constructing the Cantor set in reverse}

We describe another inductive way to construct the set of discs $\cD_n = \{D_\ii : |\ii| = n\}$. In \Cref{section:definition}, we gave a recursive construction using a tree $\{D_\ii\}_{|\ii| \leq n}$ of discs of height $n$, starting from the root. Here, we will give a different recursive construction. The output will be $(\cS_\ii)_{|\ii| \leq n}$, where each $\cS_\ii$ is a set of discs rather than a single disc. We will start with the leaves and end with the root, and $\cS_\varnothing$ will be a scaled copy of $\cD_n$.

For each leaf $\ii$ ($|\ii|=n$), we set $\cS_\ii = \{B(0,1)\}$. Now for non-leaves, we recursively define $\cS_\ii$ in terms of its children:
\begin{align}
\label{eq:S_ii recursive relation}
\cS_\ii 
= 
\{ P_{\ii j} \circ  D 
: 
j \in [4],  D \in \cS_{\ii j}
\}
\qquad\text{for } |\ii| = k < n
.
\end{align}
See \Cref{figure:reverse construction}. Comparing this with \eqref{eq:def D_i non-recursive}, it follows that
\begin{align*}
    \cD_n = \{B(0,r_0) \circ D : D \in \cS_\varnothing\}
\end{align*}
so
\begin{align}
    \label{eq:K_n and S_varnothing}
    K_n = \cup \cD_n = \{r_0 x : x \in \cup\cS_\varnothing\} =:  r_0 (\cup \cS_\varnothing). 
\end{align}
\begin{remark}
Note that $\cS_\ii$ and $D_\ii$ are very different from each other. They are related by the following: For each $k$,
\begin{align*}
\cD_n
=
\{
B(0,r_0) \circ D_\ii \circ D
:
|\ii| = k,
D \in \cS_\ii
\}
.
\end{align*}
(This is not needed for the argument.)
\end{remark}

\begin{figure}[h]
\centering
\begin{subfigure}[c]{0.6\textwidth}
    \centering
    \includegraphics[page=1,width=0.23\textwidth]{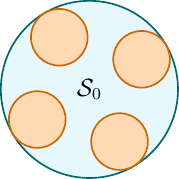}
    \includegraphics[page=2,width=0.23\textwidth]{reverse_construction1.pdf}
    \includegraphics[page=3,width=0.23\textwidth]{reverse_construction1.pdf}
    \includegraphics[page=4,width=0.23\textwidth]{reverse_construction1.pdf}
\end{subfigure}
\begin{subfigure}[c]{0.35\textwidth}
    \includegraphics[page=5,width=\textwidth]{reverse_construction1.pdf}
\end{subfigure}
\caption{Illustration of \eqref{eq:S_ii recursive relation} in the case $n=2$ and $\ii = \varnothing$.}
\label{figure:reverse construction}
\end{figure}

\subsection{Reduction to a recurrence relation}

Fix a random model (from one of the two choices; see \Cref{section:random models}). Consider the quantity $\E |\proj_\theta \cup \cS_\ii|$ for some $\theta$ and some $\ii$. By rotational symmetry, it does not depend on $\theta$. By independence and identically distributed, it only depends on $|\ii|$. Thus, we can define
\begin{align}
\label{eq:def a_k}
a_k
=
\E |\proj_\theta \cup \cS_\ii|,
\text{ where } |\ii| = k.
\end{align}
The main result of this section is the following.

\begin{lemma}
\label{lemma:derive-recurrence}
There exists a constant $A \in (0, \frac{1}{4})$ such that
\begin{align}
\label{eq:lemma derive-recurrence}
a_{k} \leq 4 \frac{r_{k+1}}{r_k} a_{k+1} - A  \left(\frac{r_{k+1}}{r_k}a_{k+1}\right)^2
\qquad
\text{for every } k = 0, \ldots, n-1
\end{align}
\end{lemma}

Note that the recurrence \eqref{eq:lemma derive-recurrence} starts at $k=n$ and ends at $k=0$. To prove \eqref{eq:lemma derive-recurrence}, we need the following, which is essentially a change of variables and an application of Fubini.

\begin{lemma}
\label{lemma:change-of-variables}
Fix $\alpha$ and $r > 0$ such that $P(r,0) \cap P(r,\alpha)  = \emptyset$. Then for any $E_1, E_2 \subseteq B(0,1)$ and any $\theta$,
\begin{align}
\label{eq:change-of-variables}
    \int_0^{2\pi}
    |\proj_\theta \discmap{P(r,\omega)} (E_1) \cap \proj_\theta \discmap{P(r,\omega+\alpha)} (E_2)|
    \, d\omega
    \geq \frac{r^2}{2(1-r)}
    |\proj_\theta E_1| |\proj_\theta E_2|
    \,.
\end{align}
\end{lemma}

\begin{proof}
By rotation symmetry, we may assume $\theta = 0$. Define $F_i = r \proj_0 E_i$.  Recalling that $\discmap{P(r,\omega)}(z)=rz+(1-r)e^{i\omega}$ we have
\begin{align*}
    &\int_0^{2\pi}
    |\proj_0 \discmap{P(r,\omega)} (E_1) \cap \proj_0 \discmap{P(r,\omega+\alpha)} (E_2)|
    \, d\omega
    \\
    &=
    \int_0^{2\pi}
    |(F_1 + (1-r)\cos\omega) \cap (F_2 + (1-r)\cos(\omega+\alpha))|
    \, d\omega
    \\
    &=
    \int_0^{2\pi}
    |F_1 \cap (F_2 + (1-r)(\cos(\omega+\alpha) - \cos\omega))|
    \, d\omega
    .
\end{align*}
Next, we will make the change of variables $t = \phi(\omega)$, where $\phi : [0,2\pi] \to \R$ is given by
\begin{align}
\label{eq:def phi(omega)}
\phi(\omega) 
= 
(1-r)(\cos(\omega+\alpha) - \cos\omega)
=
-2(1-r) \sin \frac{\alpha}{2} \sin \left(\omega + \frac{\alpha}{2}\right),
\end{align}
giving us a Jacobian factor of 
\begin{align*}
    J(t) = 
    \sum_{\omega \in [0,2\pi] : t = \phi(\omega)} \frac{1}{|\phi'(\omega)|}
    .
\end{align*}
Note that $|\phi'(\omega)| \leq 2(1-r)$ for all $\omega$, so 
\begin{align}
\label{eq:J(t) lower bound}
    J(t) \geq \frac{1}{2(1-r)}
    \qquad\text{for all }
    t \in \phi([0,2\pi])
    .
\end{align}
Furthermore, since $F_1,F_2 \subset [-r,r]$ and $P(r,0) \cap P(r,\alpha)  = \emptyset$, we have
\begin{align}
\label{eq:intersection nonempty subset image}
\begin{aligned}
    \{t \in \R : F_1 \cap (F_2 + t) \neq \emptyset\} 
    &\subset 
    [-2r,2r]
    \\
    &\subset
    [-2(1-r)|\sin\tfrac{\alpha}{2}|,2(1-r)|\sin\tfrac{\alpha}{2}|]
    \stackrel{\eqref{eq:def phi(omega)}}{=}
    \phi([0,2\pi])
    .
\end{aligned}
\end{align}
Thus, making the change of variables,
\begin{align*}
    &\int_0^{2\pi}
    |F_1 \cap (F_2 + (1-r)(\cos(\omega+\alpha) - \cos\omega))|
    \, d\omega
    \\
    &=
    \int_\R
    |F_1 \cap (F_2 + t)|
    J(t) \, dt
    \\
    &\stackrel{\eqref{eq:J(t) lower bound}}{\geq}
    \frac{1}{2(1-r)}
    \int_{\phi([0,2\pi])}
    |F_1 \cap (F_2 + t)|
    \, dt
    \\
    &
    \stackrel{\eqref{eq:intersection nonempty subset image}}{=}
    \frac{1}{2(1-r)}
    \int_\R
    |F_1 \cap (F_2 + t)|
    \, dt
    \\
    &=
    \frac{1}{2(1-r)}
    |F_1||F_2|
    .
\end{align*}
where the last equality is by Fubini's theorem. This completes the proof.
\end{proof}

\begin{proof}[Proof of \Cref{lemma:derive-recurrence}]
Let $k \in \{0,\ldots, n-1\}$. Fix some $\ii$ with $|\ii| = k$. For $j \in [4]$, define
\begin{align*}
    E_j = \cup \cS_{\ii j} \qquad\text{and}\qquad A_j = \proj_0 \discmap{P_{\ii j}}(E_{j}),
\end{align*}
where $P_{\ii j} = P(\frac{r_{k+1}}{r_{k}}, \omega_{\ii} + \frac{2\pi j}{4})$. (Recall \eqref{eq:def P_ij}.) We have
\begin{align*}
\cup
\cS_\ii
\stackrel{\eqref{eq:S_ii recursive relation}}{=}
\bigcup_{j \in [4]} 
\bigcup_{D \in \cS_{\ii j}} 
\discmap{P_{\ii j}}
(D)
=
\bigcup_{j \in [4]} 
\discmap{P_{\ii j}}
(E_{j})
,
\end{align*}
so
\begin{align}
\label{eq:proj cup S_i = cup A_j}
\proj_0
\cup
\cS_\ii
=
\bigcup_{j \in [4]} 
A_j
.
\end{align}
By the principle of inclusion--exclusion,
\begin{align}
\label{eq:cup A_j PIE}
|\bigcup_{j \in [4]} A_j|
&\leq
|A_0 \cup A_1|
+
|A_2|
+
|A_3|
=
\sum_{j \in [4]}
|A_j|
-
|A_0 \cap A_1|
\end{align}
Taking expectations, we have
\begin{align}
\label{eq:upper-bound-ak}
a_k
\stackrel{\eqref{eq:def a_k}}{=}
\E
\left|
\proj_0
\cup
\cS_\ii
\right|
\stackrel{\eqref{eq:proj cup S_i = cup A_j}}{=}
\E |\bigcup_{j \in [4]} A_j|
\stackrel{\eqref{eq:cup A_j PIE}}{\leq}
\sum_{j \in [4]}
\E|A_j|
-
\E|A_{0} \cap A_{1}|
\end{align}
Using the definition of $a_{k+1}$, we have for each $j \in [4]$,
\begin{align*}
\E|A_j|
=
\frac{r_{k+1}}{r_k} \E|\proj_0 E_j|
=
\frac{r_{k+1}}{r_k} a_{k+1}
.
\end{align*}

For the remaining term $\E|A_{0} \cap A_{1}|$, we take the expectation in two steps. First, we condition on $(\omega_\jj)_{|\jj| > k}$ and take the expectation only with respect to $\omega_\ii$. Note that $E_0$ and $E_1$ are determined by $(\omega_\jj)_{|\jj| > k}$. By \Cref{lemma:change-of-variables} (with $\alpha = \pi/2$ and $r = r_{k+1}/r_k$), 
\begin{align*}
&\E_{\omega_\ii} |A_0 \cap A_1|
\\
&=
\frac{1}{2\pi}
\int_0^{2\pi}
|\proj_0 \discmap{P(r,\omega)} (E_0) \cap \proj_0 \discmap{P(r,\omega+\alpha)} (E_1)|
\, d\omega
 \\
&\stackrel{\eqref{eq:change-of-variables}}{\gtrsim}
\left(\frac{r_{k+1}}{r_k}\right)^2
|\proj_0 E_{1}| |\proj_0 E_{2}|
\end{align*}
Now we take the expectation with respect to the remaining $\omega_\jj$. We split into two cases, depending on the random model under consideration. For the synchronized rotations model, $E_{0} = E_{1}$, so
\begin{align*}
    \E|\proj_0 E_{0}| |\proj_0 E_{1}|
    =
    \E|\proj_0 E_{0}|^2
    \geq
    (\E|\proj_0 E_{0}|)^2
    \stackrel{\eqref{eq:def a_k}}{=}
    a_{k+1}^2
    .
\end{align*}
In the independent rotations model, $E_{0}$ and $E_{1}$ are independent, so
\begin{align}
\label{eq:E_1 E_2 independent}
    \E|\proj_0 E_{0}| |\proj_0 E_{1}|
    =
    \E|\proj_0 E_{0}| \E |\proj_0 E_{1}|
    \stackrel{\eqref{eq:def a_k}}{=}
    a_{k+1}^2
    .
\end{align}
Combining the displayed equations above (starting with \eqref{eq:upper-bound-ak}) gives \eqref{eq:lemma derive-recurrence}, as desired.
\end{proof}

\subsection{Solving the recurrence relation}

\begin{proposition}\label{prop:sequence_inequality}
    Let $A > 0$. Assume that $(\delta_k)_{k\ge 0}$ and $(b_k)_{k\ge0}$ are sequences of positive numbers satisfying
    \[
    b_{k+1} \leq \delta_{k} b_{k} - A \delta_{k}^2 b_{k}^2\qquad\text{for every $k$}.
    \]
    Then
    \begin{align}
        \label{eq:b_k bound}
        b_{k}\leq \frac{I_k }{b_0^{-1} + A \sum_{j=1}^{k} I_j},
    \qquad \text{where } I_j:=\prod_{i=0}^{j-1} \delta_i.
    \end{align}
\end{proposition}

We will analyze a closely related recurrence, which is explicitly solvable, and then compare it to the previous one.

\begin{lemma}\label{prop:sequence_equality}
    Let $A>0$. Assume that $(\delta_k)_{k\ge 0}$ and $(c_k)_{k\ge0}$ are sequences of positive numbers satisfying
    \begin{equation}\label{eq:sequence_equality}
    c_{k+1} = \delta_{k} c_{k} - A \delta_{k} c_{k}c_{k+1}\qquad\text{for every $k$}.
    \end{equation}
    Then
    \[
    c_{k}= \frac{I_k }{c_0^{-1}+A  \sum_{j=1}^{k} I_j},
    \qquad \text{where } 
    I_j:=\prod_{i=0}^{j-1} \delta_i.
    \]
\end{lemma}

\begin{proof}
    By a quick computation, \eqref{eq:sequence_equality} is equivalent to
    \[
    \delta_{k} c_{k+1}^{-1} - c_k^{-1} = A\delta_{k}.
    \]
    We adopt the convention $I_0=1$. Multiplying both sides of the equation above by $I_k$, and observing that $I_{k+1}=I_k \delta_{k}$, we obtain the relation
    \[
    I_{k+1}c_{k+1}^{-1}-I_{k}c_{k}^{-1}=A I_{k+1}, \qquad k\ge 0.
    \]
    It follows that
    \[
    I_k c_k^{-1} = I_0 c_0^{-1}+\sum_{j=1}^{k} (I_{j}c_{j}^{-1}-I_{j-1}c_{j-1}^{-1})=c_0^{-1}+A\sum_{j=1}^{k} I_{j},
    \]
    as desired.
\end{proof}

\begin{proof}[Proof of Proposition \ref{prop:sequence_inequality}]
    We start observing that if $(b_k)$ satisfies
    \[
    b_{k+1}\le \delta_{k} b_k-A\delta_{k}^2 b_k^2,
    \]
    then since $b_{k+1}\le \delta_{k} b_{k}$, it follows that $(b_k)$ also satisfies
    \begin{align}
    \label{eq:sequence_inequality_proof}
    b_{k+1} &\le \delta_{k} b_k-A \delta_{k} b_{k}b_{k+1}
    .
    \end{align}
    Next we claim a comparison property: if $(b_k)_{k=0}^\infty$ satisfies \eqref{eq:sequence_inequality_proof} and $(c_k)_{k=0}^\infty$ satisfies
    \begin{align}
    \label{eq:sequence_equality_proof}
    c_{k+1} &= \delta_{k} c_k-A\delta_{k}c_{k}c_{k+1}
    \end{align}
    and moreover $b_0=c_0$, then $b_k\le c_k$ for every $k$. We prove this by induction on $k$. The base case $k=0$ is true by assumption. For the inductive step, suppose that $b_k\le c_k$. Then using the fact that $g(t)=t/(1+A t)$ is increasing for $t\ge 0$ (since its derivative is $g'(t)=(1+A t)^{-2}$), we have 
    \[
    b_{k+1} \stackrel{\eqref{eq:sequence_inequality_proof}}
    \le 
    \frac{\delta_{k} b_k}{1+A \delta_{k} b_k}
    \le
    \frac{\delta_{k} c_k}{1+A \delta_{k} c_k}
    \stackrel{\eqref{eq:sequence_equality_proof}}
    =
    c_{k+1}
    .
    \]
    This completes the induction, and since we can solve for $c_k$ explicitly by Proposition \ref{prop:sequence_equality}, we reach the conclusion.    
\end{proof}

\subsubsection{Relation with Bernoulli's equation}
We briefly comment on the reason that made us consider the modified sequence of Proposition \ref{prop:sequence_equality}. We started from the observation that the recurrence relation
\[
y_{k+1}=\delta_{k+1} y_k-Ay_k^2
\]
can be rewritten as
\[
y_{k+1}-y_k=(\delta_{k+1}-1) y_k-Ay_k^2,
\]
and interpreting the left-hand side as a discrete derivative this resembles a discretized version of (a particular case of) Bernoulli's differential equation
\[
y'(x)=P(x)y(x)-A y(x)^{2},
\]
where the discrete integer variable $k$ has been replaced by the continuous real variable $x$.
The usefulness of this point of view is that this equation is explicitly solvable. First one can substitute $u(x)=y(x)^{-1}$ to obtain the equation
\begin{equation}\label{eq:u}
u'(x)+u(x)P(x)=A.
\end{equation}
Now one looks for a function $I(x)$ (an ``integrating factor'') with the property that 
\begin{equation}\label{eq:I}
    I'(x)=I(x)P(x).
\end{equation}
Indeed then multiplying \eqref{eq:u} by $I(x)$ one deduces
\begin{equation}\label{eq:Iu}
AI(x)=I(x)u'(x)+u(x)I(x)P(x)=(I(x)u(x))'
\end{equation}
and hence one can integrate this to solve for $I(x)u(x)$. Now the discretized version of \eqref{eq:I} is
\[
I_{k+1}-I_k=I_k(\delta_{k+1}-1),
\]
which yields $I_k=\prod_{i=1}^k \delta_i$. The discretized version of \eqref{eq:Iu} thus becomes
\[
I_{k+1}u_{k+1}-I_k u_k=A I_k.
\]
If one substitutes $u_k=y_k^{-1}$ then it is straightforward to show that this reduces to the equation
\[
y_{k+1}=\delta_{k+1}y_k-Ay_k y_{k+1}.
\]
So by the reasoning above this recurrence relation is explicitly solvable, but on the other hand the original recurrence relation is close enough to this one that one can compare the two, as shown in Proposition \ref{prop:sequence_inequality}.

\subsection{Proof of the main theorem}

\begin{proof}[Proof of the upper bounds in \Cref{thm:main}]
We start with \eqref{eq:E Fav Kn main thm}. Fix $n$. Let $a_0, \ldots, a_n$ be defined as in \eqref{eq:def a_k}. By \Cref{lemma:derive-recurrence}, we can apply \Cref{prop:sequence_inequality} with $b_k = a_{n-k}$ and $\delta_k = 4\frac{r_{n-k}}{r_{n-k-1}}$. This gives us
\begin{align}
    \label{eq:a_0 bound step}
    a_0 = b_n
    \stackrel{\eqref{eq:b_k bound}}{\lesssim}
    \frac{I_n }{1+\sum_{j=1}^{n} I_{j}}
    ,
    \qquad\text{where } I_j=\prod_{i=0}^{j-1} \delta_i = 4^j \frac{r_n}{r_{n-j}}.
\end{align}
Thus,
\begin{align}
    \label{eq:goal a_0 lesssim}
    \E |\proj_\theta (K_n)|
    \stackrel{\eqref{eq:K_n and S_varnothing}, \eqref{eq:def a_k}}=
    r_0 a_0
    \stackrel{\eqref{eq:a_0 bound step}}{\lesssim}
    \left(
    \sum_{j=0}^{n}\frac{1}{4^j r_j}
    \right)^{-1}
    .
\end{align}
Integrating in $\theta$ gives the upper bound in \eqref{eq:E Fav Kn main thm}.

For the upper bound in \eqref{eq:E Fav K(r) main thm}, suppose $r \in [r_n, r_{n-1}]$. Let $\wt K_n$ be as in the statement of \Cref{lemma:K(r) upper bound K_n}. Then by \Cref{lemma:K(r) upper bound K_n},
\begin{align*}
    \E \Fav(K(r))
    &\lesssim
    \E \Fav(\wt K_n)
    \\
    &\stackrel{\eqref{eq:E Fav Kn main thm}}{\lesssim}
    \left(
        \sum_{j=0}^{n-1}\frac{1}{4^m r_m}
        +
        \frac{1}{4^n \min(r, \frac{1}{3}r_n)}
    \right)^{-1}
    \approx
    \left(
        \sum_{j=0}^{n-1}\frac{1}{4^m r_m}
        +
        \frac{1}{4^n r}
    \right)^{-1}
\end{align*}
\end{proof}

\section{From expected decay to almost-sure decay}
\label{section:expected to almost sure decay}

The goal of this section is to prove \Cref{theorem:expected to almost sure decay}. For this section, we work only with the fully independent random model. Fix an admissible sequence $r_0 \geq r_1 \geq r_2 \cdots$ (recall \Cref{definition:admissible}).  
For $r \in [r_n, r_{n-1})$, define 
\[
g(r) = \sum_{j=0}^{n-1} \frac{1}{4^j r_j} + \frac{1}{4^n r}
\]
and recall that $\E \Fav K(r) \approx \frac{1}{g(r)}$ (\Cref{thm:main}). 

We say $r \in [r_n, r_{n-1})$ is \emph{trivial} if one of the $n+1$ terms in the sum in the definition of $g$ dominates, i.e.,
\begin{align}
\label{eq:r trivial}
\max(\{\frac{1}{4^j r_j} : j = 0, 1, \ldots, n-1\} \cup \{\frac{1}{4^n r}\}) \geq \frac{1}{4} g(r).
\end{align}
The reason we call such values ``trivial'' is because, as the following lemma states, they imply the desired Favard length bound deterministically.

\begin{lemma}
\label{lemma:trivial Fav 1/g bound}
    If $r$ is trivial, then $\Fav(K(r)) \lesssim \frac{1}{g(r)}$. The implied constant is absolute.
\end{lemma}

\begin{proof}
    For any $m$, $K$ is contained in the union of $4^m$ discs of radius $r_m$. Thus, $K(r)$ is contained in the union of $4^m$ discs of radius $r_m+r$. This implies
    \begin{align}
    \label{eq:Fav K(r) trivial upper bound}
    \Fav(K(r)) \lesssim 4^m \max(r_m,r) \qquad\text{for any $r$ and any $m$}
    \end{align}
    Now suppose $r \in [r_n, r_{n-1})$ is trivial. Then by applying \eqref{eq:Fav K(r) trivial upper bound} with $m=0, 1, \ldots, n$, we get
    \[
    \Fav(K(r))
    \stackrel{\eqref{eq:Fav K(r) trivial upper bound}}{\lesssim} 
    \min (\{4^j r_j : j = 0, \ldots, n-1\} \cup \{4^n r\})
    \stackrel{\eqref{eq:r trivial}}{\lesssim}
    \frac{1}{g(r)}
    \]
    as desired
\end{proof}

For $m \geq 0$, let $K^{(m)} = \bigcap_{n=0}^\infty K^{(m)}_n$ be the random Cantor set generated with the radii $(r_m, r_{m+1}, r_{m+2}, \ldots)$ instead of $(r_0, r_1, r_2, \ldots)$. We make two observations. First,
\begin{align}
\label{eq:union-copies}
    \text{$K(r)$ is a union of $4^m$ independent copies of $K^{(m)}(r)$.}
\end{align}
(Recall that in this section, we work only with the fully independent random model.) Second, by \Cref{thm:main},
\begin{align}
\label{eq:expectation K^(m)}
4^m \E \Fav K^{(m)}(r)
\approx
\left(\sum_{j=m}^{n-1} \frac{1}{4^j r_j} + \frac{1}{4^n r} \right)^{-1}
=
\frac{1}{g(r) - g(r_{m-1})}
\end{align}

\begin{lemma}[Hoeffding]
\label{lemma:hoeffding-corollary}
For all $m \in \N, r > 0, s > 0$,
\begin{align*}
    \P\bigl[\Fav K(r) > 4^m \E \Fav K^{(m)}(r) + s \bigr]
    \leq
    \exp\left(-\frac{s^2}{2 \cdot 4^m (r_m+r)^2}\right)
    .
\end{align*}
\end{lemma}

\begin{proof}
We adapt the argument from the proof of \cite[Lemma 3.5]{chang-shmerkin-suomala}. Fix $r > 0$. By \eqref{eq:union-copies} and subadditivity of Favard length, it follows that $\Fav K(r)$ is bounded above by the sum $Y$ of $4^m$ independent copies of $\Fav K^{(m)}(r)$. Since $K^{(m)}$ is contained in a disc of radius $r_m$, we have $0\le \Fav K^{(m)}(r) \leq 2(r_m+r)$, so Hoeffding's inequality (e.g., \cite[Theorem 2.2.1]{vershynin}) yields (for all $s > 0$)
\begin{equation*}
    \P\bigl[\Fav(K(r))> 4^m \E \Fav K^{(m)}(r) + s \bigr]
    \leq
    \P\bigl[Y>\E[Y] + s \bigr]
    \leq
    \exp\left(-\frac{s^2}{2 \cdot 4^m (r_m+r)^2}\right)
    . 
    \qedhere
\end{equation*}    
\end{proof}

\begin{lemma}[Nontrivial implies existence of good scale]\label{lem:nontrivial}
Suppose $r \in [r_n, r_{n-1})$ is non-trivial (recall \eqref{eq:r trivial}). Then there exists $m \leq n-1$ such that $g(r_m) \in [\frac{1}{4}g(r), \frac{3}{4} g(r)]$ and
\begin{align}
\label{eq:prob using Hoeffding}
    \P\bigl[\Fav K(r) > \frac{C}{g(r)} \bigr]
    \leq
    \exp\left(- c \cdot 2^m \right)\,.
\end{align}
The constants $C, c > 0$ are absolute.
\end{lemma}

\begin{proof}
Since $r$ is nontrivial, the consecutive terms of the finite sequence 
\[
0, g(r_0), g(r_1), \ldots, g(r_{n-1}), g(r)
\]
differ by at most $\frac{1}{4} g(r)$. Thus, there exists $m \leq n-1$ such that 
\begin{align}
    \label{eq: g m-1 m r comparable}
    g(r_{m-1}), g(r_m) \in [\frac{1}{4} g(r), \frac{3}{4} g(r)].
\end{align}

We keep this $m$ fixed for the remainder of this proof. Note that 
\begin{align}
\label{eq:expectation K^(m) comparable 1/g(r)}
4^m \E \Fav K^{(m)}(r)
\stackrel{\eqref{eq:expectation K^(m)}}{\approx}
\frac{1}{g(r) - g(r_{m-1})}
\stackrel{\eqref{eq: g m-1 m r comparable}}{\approx} 
\frac{1}{g(r_m)}
.
\end{align}
Now we apply \Cref{lemma:hoeffding-corollary} with $m$ as above and with $s = \frac{1}{g(r) - g(r_{m-1})}$. Combining this with \eqref{eq:expectation K^(m) comparable 1/g(r)} gives us
\begin{align*}
    \P\bigl[\Fav K(r) > \frac{C}{g(r)} \bigr]
    \leq
    \exp\left(-\frac{c}{g(r_m)^2 \cdot 4^m \cdot r_m^2}\right)\,.
\end{align*}
To complete the proof, we need to show
\begin{align}
    \label{eq:g^2 4^m r_m^2}
    g(r_m)^2 \cdot 4^m \cdot r_m^2 \lesssim \frac{1}{2^m}
\end{align}
By admissibility, we have
\[
g(r_m) 
= 
\sum_{j=0}^m \frac{1}{4^j r_j} 
\stackrel{\eqref{eq:admissible}}{\leq}
\sum_{j=0}^m \frac{1}{4^j \cdot 3^{m-j} r_m} 
=
\frac{1}{3^m r_m} \sum_{j=0}^m \left(\frac{3}{4}\right)^j 
\approx 
\frac{1}{3^m r_m},
\]
which rearranges to $g(r_m)^2 \cdot 4^m \cdot r_m^2 \lesssim (\frac{4}{9})^m$. This proves \eqref{eq:g^2 4^m r_m^2}.
\end{proof}

\begin{proof}[Proof of \Cref{theorem:expected to almost sure decay}]
By monotonicity of $r \mapsto \Fav K(r)$, it suffices to show 
\begin{align}
\label{eq:limsup delta_k}
    \limsup_{k \to \infty} \frac{\Fav K(\delta_k)}{2^{-k}} \leq C
   \qquad\text{almost surely}.
\end{align}
where $(\delta_k)_{k=1}^\infty$ is a sequence such that $\E \Fav K(\delta_k) \approx 2^{-k}$. (Such a sequence exists since $r \mapsto \E \Fav K(r)$ is continuous and $\lim_{r \to 0} \E \Fav(K(r)) = 0$.)

If $\delta_k$ is nontrivial, then by \Cref{lem:nontrivial}, there exists $m_k$ such that 
\begin{align}
\label{eq:g(r_m_k)}
g(r_{m_k}) \in \left[\frac{1}{4}g(\delta_k), \frac{3}{4} g(\delta_k)\right]
\end{align} 
and
\begin{align*}
    \P\bigl[\Fav K(\delta_k) > \frac{C}{g(\delta_k)} \bigr]
    \leq
    \exp\left(-c \cdot 2^{m_k}\right)\,.
\end{align*}
Note that $g(r_{m_k}) \approx g(\delta_k) \approx 2^k$, so we know $\# \{k : m_k = m\} = O(1)$ for each integer $m$. Thus, by making $C$ larger if needed,
\[
\sum_{k=1}^\infty
\P\bigl[\Fav K(\delta_k) > \frac{C}{g(\delta_k)} \bigr]
\stackrel{\textup{Lem } \ref{lemma:trivial Fav 1/g bound}}{=}
\sum_{k : \delta_k \text{ nontrivial}}
\P\bigl[\Fav K(\delta_k) > \frac{C}{g(\delta_k)} \bigr]
< \infty.
\]
By Borel--Cantelli,
\begin{align*}
    \limsup_{k \to \infty} \frac{\Fav K(\delta_k)}{1/g(\delta_k)} \leq C
   \qquad\text{almost surely}.
\end{align*}
Combining this with \Cref{thm:main} gives \eqref{eq:limsup delta_k}.
\end{proof}

\begin{remark}
\label{remark:lim=0 necessary}
    The assumption $\lim_{r \to 0} \E \Fav(K(r)) = 0$, or equivalently $\sum_{j=0}^\infty 4^{-j} r_j^{-1}=\infty$, is necessary in Theorem \ref{theorem:expected to almost sure decay}. To see this, fix some large $n\in \N$ and consider the admissible sequence 
    \[
    r_j
    =
    \begin{cases}
        4^{-j} &\text{for } 0\le j\le n
        \\
        4^{-n}3^{n-j} &\text{for } j\ge n
        .
    \end{cases} 
    \]
    The corresponding Cantor set resembles the classical four corner Cantor set up to scale $4^{-n}$, except for the rotations. At smaller scales the radii decay so slowly that the Favard length essentially freezes at scale $4^{-n}$: it follows that for $0<r<4^{-n}$, if $k\ge n+1$ is such that $r \in [r_k, r_{k-1}]$,
    \begin{equation}
    \label{eq:lim=0 necessary E}
    \E[\Fav(K(r))]
    \stackrel{\eqref{eq:E Fav K(r) main thm}}{\approx}
    \left(
    n
    + \sum_{j=1}^{k-n-1}\left(\frac{3}{4}\right)^{j} +
    \frac{1}{4^k r}
    \right)^{-1}\approx n^{-1}.
    \end{equation}

    In fact, the Favard length does not only freeze in expectation: $\Fav(K) \approx \Fav K_n$ deterministically. This follows from the Vitali covering lemma on $\R$ applied to the projections of the discs in $K_n$ together with the observation $\Fav(K\cap D)\approx 4^{-n}$ for all the discs $D \subset K_n$ (this follows from the capacity lower bound \eqref{eq:Cap K_n lower bd} applied to $K\cap D$). Thus, for all $0 < r < 4^{-n}$, we have 
    \[
    \Fav(K(r)) \approx \Fav(K) \approx \Fav K_n.
    \]

    Now, let $\wt K_n$ be the $n$th stage of the standard four corner Cantor set with no rotations (\Cref{fig:four corner Cantor discs}). By continuity, there exists some $\epsilon > 0$ ($\epsilon\ll 4^{-n}$) such that if the rotations $\omega_\ii$ satisfy $\omega_\ii\in [0,\varepsilon]$ for all $|\ii| \leq n$, then 
    \[
    \Fav K_n \approx \Fav \wt K_n
    \stackrel{\eqref{eq:bateman-volberg}}{\gtrsim} \frac{\log n}{n} 
    .
    \]
    In the independent rotations model, the event $\omega_\ii\in [0,\varepsilon]$ for all $|\ii| \leq n$ occurs with positive probability. Combined with the previous paragraph, this shows that with positive probability,
    \[
    \Fav(K(r)) \gtrsim \frac{\log n}{n} \qquad\text{for all $0 < r < 4^{-n}$}.
    \]
    
    Comparing this with \eqref{eq:lim=0 necessary E}, we see that with positive probability
    \[
    \liminf_{r \to 0} \frac{\Fav K(r)}{\E \Fav K(r)} 
    \gtrsim 
    \log n
    \]
    so \eqref{eq:almostsure} fails.
\end{remark}

\section{Prescribing Favard length decay}
\label{section:prescribe}

Let $r_0 \geq r_1 \geq \cdots$ be admissible. (Recall \Cref{definition:admissible}.) In this section, it will be more convenient to work with the reciprocals, so we define $s_n = 1/r_n$ and we work with the variable $s = 1/r$.

Define $\fjump : [0, \infty) \to [0, \infty)$ by $\fjump(0) = 0$ and 
\begin{align}
\label{eq:def fjump}
\fjump(s) 
=
\sum_{j=0}^{n-1} \frac{s_j}{4^j}   +    \frac{s}{4^n}\quad\quad \text{for $s_{n-1} <s\le s_n$.}
\end{align}
Note that $1/\fjump(1/r)$ is precisely the right-hand side of \eqref{eq:E Fav K(r) main thm}. Thus, in view of \Cref{thm:main}, to prove \Cref{theorem:prescribed}, it suffices to show the following.

\begin{proposition}
\label{prop:1/f(1/r) to s_n}
Let $f:(0,\infty)\to (0,\infty)$ be a concave increasing function satisfying $\lim_{r \to \infty} f'(r) = 1$.
Then there exists a sequence $0 = s_{-1} < s_0 < s_1 < s_2 < \ldots$ such that $\frac{s_{n+1}}{s_{n}}\ge 3$ and such that $f(r) \approx 1/\fjump(1/r)$.
\end{proposition}

To prove \Cref{prop:1/f(1/r) to s_n}, we will rely on the following two lemmas.

\begin{lemma}
\label{lemma:1/f(1/x)}
For any concave increasing function $f : (0, \infty) \to (0, \infty)$, there exists a concave increasing function $g : (0, \infty) \to (0, \infty)$ such that $g(x) \leq 1/(f(1/x)) \leq 2 g(x)$
\end{lemma}

\begin{lemma}
\label{lemma:from f to s_n}
Let $f:[0,\infty)\to [0,\infty)$ be a concave increasing function satisfying $f(0)=0$ and $f'(0) = 1$. (In particular, $f$ is continuous at $0$.) Then there exists a sequence $0 = s_{-1} < s_0 < s_1 < s_2 < \ldots$ such that $\frac{s_{n+1}}{s_{n}}\ge 3$ and such that $f(s) \approx \fjump(s)$. The implied constant is absolute.
\end{lemma}

\begin{proof}[Proof of \Cref{prop:1/f(1/r) to s_n} from Lemmas \ref{lemma:1/f(1/x)} and \ref{lemma:from f to s_n}]
Let $f:[0,\infty)\to [0,\infty)$ be a concave increasing function satisfying $\lim_{r \to \infty} f'(r) = 1$. Then by \Cref{lemma:1/f(1/x)}, there exists a concave increasing function $g : (0, \infty) \to (0, \infty)$ such that $g(r) \approx 1/f(1/r)$. The condition $\lim_{r \to \infty} f'(r) = 1$ implies that $g(0) = 0$ and $g'(0) = 1$. Thus, we may apply \Cref{lemma:from f to s_n} to $g$ to get a sequence $0 = s_{-1} < s_0 < s_1 < s_2 < \ldots$ such that $\frac{s_{n+1}}{s_{n}}\ge 3$ and $g(s) \approx \fjump(s)$. Thus, we have $f(r) \approx 1/(g(1/r)) \approx 1/\fjump(1/r)$, as desired.
\end{proof}

The remainder of this section will be the proofs of \Cref{lemma:1/f(1/x)} and \Cref{lemma:from f to s_n}. For that, we will need some auxiliary lemmas.

\begin{lemma}
\label{lemma:increasing star shaped}
Let $\cC$ be the set of all functions $\phi : (0, \infty) \to (0, \infty)$ such that $\phi(x)$ is increasing and $\phi(x)/x$ is non-increasing. Then:

\begin{enumerate}
    \item If $\phi : (0, \infty) \to (0, \infty)$ is increasing and concave, then $\phi \in \cC$.
    \item If $\phi \in \cC$, then $\psi(x) := 1/\phi(1/x)$ satisfies $\psi \in \cC$.

    \item If $\phi \in \cC$, then 
    \[
    (1-t) \phi(x) + t\phi(y) \leq 2 \phi((1-t)x+ty) \qquad\text{for all $x,y \in (0,\infty)$ and $t \in [0,1]$}
    \]
\end{enumerate}
\end{lemma}

Property (3) can be viewed as a ``partial converse'' to (1): functions in $\cC$ are ``concave up to a factor of $2$.''

\begin{proof}
For (1), note that if $\phi : (0, \infty) \to (0, \infty)$ is increasing and concave, then we can extend $\phi$ continuously to $0$, with $\phi(0) \geq 0$. By concavity applied to $0 < x < y$, we have $\phi(x) \geq \frac{x}{y} \phi(y) + (1-\frac{x}{y})\phi(0) \geq \frac{x}{y} \phi(y)$. This proves (1). (2) is immediate from the observations that $x \mapsto 1/x$ is decreasing and  $\frac{\psi(x)}{x} = (\frac{\phi(1/x)}{1/x})^{-1}$.  For (3), let $x,y \in (0, \infty)$ and $t \in [0,1]$. Without loss of generality, assume $x \leq y$, so that $x \leq (1-t)x + ty \leq y$. Since $\phi$ is increasing,
\[
(1-t) \phi(x) \leq \phi(x) \leq \phi((1-t)x + ty)
.
\]
Since $z \mapsto \phi(z)/z$ is non-increasing,
\[
t\phi(y) 
\leq
\frac{ty}{(1-t)x + ty} \phi((1-t)x + ty)
\leq
\phi((1-t)x + ty)
\]
Combining these gives (3).
\end{proof}

\begin{proof}[Proof of \Cref{lemma:1/f(1/x)}]
Let  $f : (0, \infty) \to (0, \infty)$ be a concave increasing function. By \Cref{lemma:increasing star shaped}(1), $f \in \cC$. By \Cref{lemma:increasing star shaped}(2), $h(x) := 1/f(1/x)$ satisfies $h \in \cC$. Now define the upper concave envelope of $h$:
\[
g(z) = \sup\{ (1-t)h(x) + t h(y) : x,y \in (0,\infty), t \in [0,1], (1-t)x + ty = z\}.
\]
It is clear that $h(x) \leq g(x)$. Furthermore, by \Cref{lemma:increasing star shaped}(3), $g(x) \leq 2h(x)$. This proves the lemma (after redefining $g$ by a factor of $2$).
\end{proof}

\begin{lemma}
\label{lemma:wlog f C^1}
If $f : (0, \infty) \to (0, \infty)$ is concave and increasing, then there exists a $g : (0, \infty) \to (0, \infty)$ which is concave, increasing, and $C^1$, and such that $f(s) \approx g(s)$.
\end{lemma}

\begin{proof}
Let $\phi$ be a smooth bump function supported in $(1,2)$ with $\int_1^2 \phi(t) \, dt = 1$. Take $g(x) = \int_1^2 f(xt) \phi(t) \, dt$. It is clear that $g$ is concave, increasing, and $C^1$.

We claim that $f(x) \leq g(x) \leq 2f(x)$. Since $f$ is increasing, we have $f(x) \leq g(x)$. To show $g(x) \leq 2f(x)$, note that by \Cref{lemma:increasing star shaped}, if $x > 0$ and $t\geq 1$, then $\frac{f(xt)}{xt} \leq \frac{f(x)}{x}$. Thus $g(x) = \int_1^2 f(xt) \phi(t) \, dt \lesssim \int_1^2 t f(x) \, dt \leq 2f(x)$, as desired.
\end{proof}

\begin{lemma}
\label{lemma:f deriv control}
Let $f:[0,\infty)\to [0,\infty)$ be a concave $C^1$ increasing function satisfying $f(0)=0$. Then for any $1 \leq \alpha < \beta$, there exists a function $h : [0, \infty) \to [0,\infty)$, which is also a concave $C^1$ increasing function satisfying $h(0)=0$, and with the additional properties 
\[
h'(0) = f'(0), 
\qquad 
h'(s) \geq \frac{1}{\beta} h'(s/\alpha), 
\qquad\text{and}\qquad 
f(s) \leq h(s) \leq \frac{1}{1-\alpha/\beta} f(s).
\]
\end{lemma}

\begin{proof}
Define $h(s)$ by $h(0) = 0$ and
\begin{align}
\label{eq:def h prime}
h'(s)
:=
\sup_{k \geq 0} \frac{1}{\beta^k} f'\left(\frac{s}{\alpha^k}\right)
\end{align}
Since the $k$th term in the supremum is bounded above by $\frac{f'(0)}{\beta^k}$, the supremum is finite for each $s \in (0, \infty)$. Then since each term in the supremum is non-increasing, non-negative, and continuous, $h'$ is as well. 

By considering the $k=0$ and $k \geq 1$ cases of the supremum separately, we have 
\[
h'(s) \geq f'(s)
\qquad\text{and}\qquad
h'(s) \geq \frac{1}{\beta}h'(s/\alpha)
.
\]
By integrating, we have $h(s) \geq f(s)$. The definition of $h$ also implies $h'(0) =f'(0)$.

It remains to show $h(s) \leq \frac{1}{1-\alpha/\beta} f(s)$. For this, we have
\begin{align*}
h(s)
=
\int_0^s h'(t) \, dt
\stackrel{\eqref{eq:def h prime}}{\leq}
\sum_{k=0}^\infty
\int_0^s
\frac{1}{\beta^k} f'\left(\frac{t}{\alpha^k}\right)
\, dt
=
\sum_{k=0}^\infty 
\frac{\alpha^k}{\beta^k} f\left(\frac{s}{\alpha^k}\right)
\leq
\left(\sum_{k=0}^\infty 
\frac{\alpha^k}{\beta^k}\right) f(s),
\end{align*}
as desired.
\end{proof}

\begin{proof}[Proof of \Cref{lemma:from f to s_n}]
We may assume $f$ is $C^1$. By applying \Cref{lemma:f deriv control} with $\beta = 4$ and $\alpha = 3$, we may further assume that $f$ satisfies \begin{align}
\label{eq:f' lower bound}
f'(s) \geq \frac{1}{4}f'(\frac{s}{\alpha}) \qquad\text{for all } s \geq 0.
\end{align}
Since $f$ is concave and increasing, the derivative $f'$ is positive and non-increasing. We define $(s_n)_{n=0}^\infty$ as follows:
\begin{equation}
\label{eq:def s_n}
    s_n = \inf\{s \geq 0 : f'(s) \leq 4^{-(n+1)}\} \in (0, \infty].
\end{equation}
Since $f'$ is non-increasing, the sequence $s_n$ is increasing. By \eqref{eq:f' lower bound}, we also have $s_n \geq \alpha s_{n-1}$.

Note that $\fjump$ is piecewise linear with a jump of size $\frac{s_j}{4^{j+1}}$ at $s=s_j$ for every $j$. With that in mind, we define $\fnojump : [0, \infty) \to [0, \infty)$ to be the unique continuous piecewise linear function with $\fnojump(0) = 0$ and $\fnojump'(s) = 4^{-n}$ for $s \in (s_{n-1}, s_{n})$. In the language of distributions (or Lebesgue decomposition),
\begin{align}
\label{eq:jump no jump derivative}
\fjump'
=
\fnojump'
+
\sum_{j=0}^\infty \frac{s_j}{4^{j+1}} \delta_{s_j}
.
\end{align}
We claim that $\fjump(s)\approx \fnojump(s).$ Indeed, integrating \eqref{eq:jump no jump derivative} and using $\fjump(0) = \fnojump(0)$ gives, for $s \in (s_{n-1}, s_n]$,
\[
\fjump(s)
-
\fnojump(s)
=
\sum_{j=0}^{n-1} \frac{s_j}{4^{j+1}} 
=
\frac{1}{4}\fjump(s_{n-1})
\leq
\frac{1}{4}\fjump(s).
\]
This shows that $\frac{3}{4} \fjump(s) \leq \fnojump(s) \leq \fjump(s)$.

It remains to show that $f(s)\approx \fnojump(s)$. Suppose $s \in (s_{n-1},s_n)$. Then \eqref{eq:def s_n} implies $4^{-(n+1)} \leq f'(s) \leq 4^{-n}$ (for $n = 0$, the upper bound follows from $f'(0) = 1$). Thus 
\[
f'(s) \approx 4^{-n} = \fnojump'(s).
\]
Integrating this gives $f(s) \approx \fnojump(s)$.
\end{proof}

\section{Estimate of analytic capacity}\label{sec:analytic-capacity}
In this section we estimate the analytic capacity of disc-like Cantor sets, Theorem \ref{thm:anal-cap}.

We will follow the proof of analogous estimates for the deterministic non-homogeneous four corner Cantor sets, see Section 4.7 in \cite{tolsa2014analytic}. The main tool is a characterization of analytic capacity due to Tolsa \cite{tolsa1999,tolsa2002,tolsa2003acta}, which involves \emph{Menger curvature}.
Given a Radon measure $\mu$ on $\R^2$ and $x\in\R^2$ we define
\begin{equation*}
    c^2_\mu(x) \coloneqq \iint \frac{1}{R(x,y,z)^2}\, d\mu(y)d\mu(z),
\end{equation*}
where $R(x,y,z)$ is the radius of the circle passing through $x, y$ and $z$ (for collinear $x,y, z$, we set $R(x,y,z)=\infty$). We also set $c_\mu(x)\coloneqq (c^2_\mu(x))^{1/2}.$
For more information and history of Menger curvature we refer to the monograph \cite{tolsa2014analytic}. 

The following characterization of analytic capacity was shown in \cite{tolsa1999, tolsa2003acta}, see also \cite[(2.4)]{tolsa2003acta} or \cite[Corollary 4.17]{tolsa2014analytic}. The characterization involves the potential
\begin{equation*}
    U_\mu(x)\coloneqq \sup_{s>0}\theta_\mu(x,s) + c_\mu(x),
\end{equation*}
introduced by Verdera \cite{verdera2000}. Recall that $\theta_\mu(x,s)=\mu(B(x,s))/s$.
\begin{theorem}\label{thm:cap-char}
    For any compact set $E\subset\R^2$ we have
    \begin{equation*}
        \gamma(E)\approx \sup_\nu \nu(E)
    \end{equation*}
    where the supremum is taken over all Borel measures $\nu$ on $E$ such that for all $x\in\spt\nu$
    \begin{equation}\label{eq:beta-est-AT}
        U_\nu(x)\le 1.
    \end{equation}
\end{theorem}
Theorem \ref{thm:cap-char} is a very convenient tool for proving lower bounds on analytic capacity. We also have a ``dual'' estimate for the analytic capacity which is more convenient for proving upper bounds.
It was shown in \cite[Theorem 3.3]{tolsa2002}, \cite{tolsa2003acta}, see also \cite[Theorem 4.24]{tolsa2014analytic}.
\begin{theorem}\label{thm:tolsa}
    For any compact set $E\subset\R^2$ we have
    \begin{equation*}
        \gamma(E)\approx \inf_\nu \nu(\R^2)
    \end{equation*}
    where the infimum is taken over all finite Borel measures $\nu$ on $\R^2$ such that for all $x\in E$ we have
    \begin{equation}\label{eq:beta-est-AT2}
        U_\nu(x)\ge 1.
    \end{equation}
\end{theorem}

We proceed to prove \eqref{eq:analytic-cap}. Recall the notation $\theta_k=4^{-k}r_k^{-1}.$ 
For the sake of brevity, set
    \begin{equation*}
        A\coloneqq \left(\sum_{k=0}^{n}\theta_k^2\right)^{1/2}.
    \end{equation*}
We remark in passing that
\begin{equation}\label{eq:Aest}
    A\ge \theta_k,\quad\text{for } k\in\{0,\dots,n\}.
\end{equation}
Our goal is to show $\gamma(K_n)\approx A^{-1}.$ 

Recall that $\mu_n$ denotes the natural probability measure on $K_n$
\begin{equation*}
    \mu_n = \mathcal{L}^2(K_n)^{-1}\mathcal{L}^2|_{K_n}.
\end{equation*}
Consider the normalized measure 
    \begin{equation}\label{eq:nu-def}
        \nu = A^{-1} \mu_n.
    \end{equation}
We claim that
\begin{equation}\label{eq:Uofnu}
    U_\nu(x)\approx 1\quad\text{for }x\in K_n.
\end{equation}
The estimate above immediately implies Theorem \ref{thm:anal-cap}.
\begin{proof}[Proof of Theorem \ref{thm:anal-cap} using \eqref{eq:Uofnu}]
    Observe that for $0<c\le 1$
    \begin{equation*}
        U_{c\nu}(x)\le cU_\nu(x),
    \end{equation*}
    and so applying Theorem \ref{thm:cap-char} to $c\nu$ with $c$ small enough (depending on the implicit constants in \eqref{eq:Uofnu}) we get $\gamma(K_n)\lesssim c\nu(K_n)\approx A^{-1}.$ Conversely, for $C\ge 1$ we have
    \begin{equation*}
        U_{C\nu}(x)\ge CU_\nu(x),
    \end{equation*}
    so that by \eqref{eq:Uofnu} we have \eqref{eq:beta-est-AT2} for $C\nu$, assuming $C$ is large enough. Thus, Theorem \ref{thm:tolsa} implies $\gamma(K_n)\gtrsim C\nu(K_n)\approx A^{-1}.$
\end{proof}

The remainder of this section is dedicated to proving \eqref{eq:Uofnu}. We begin by showing that $\nu$ has linear growth, which follows readily from the density estimates of \Cref{lem:dens-est}.
\begin{lemma}\label{lem:dens}
    For all $x\in K_n$ and $r > 0$, we have $\theta_\nu(x,r)\lesssim 1.$ (The implied constant is absolute.)
\end{lemma}
\begin{proof}
    By \Cref{lem:dens-est}, for all $x \in K_n$ and $r > 0$, we have
    \[
    \theta_{\mu_n}(x,r)
    \lesssim
    \max\left\{\theta_0, \ldots, \theta_n\right\}
    \stackrel{\eqref{eq:Aest}}{\leq}
    A,
    \]
    so $\theta_\nu(x,r) \lesssim 1$. 
    
    
\end{proof}

A key property of sets $K_n$ is the following estimate on the curvature of $\nu$. The arguments are similar to those in the proof of \cite[Lemma 4.29]{tolsa2014analytic}. 

\begin{lemma}\label{lem:}
    For all $x\in K_n$, we have $c_\nu(x) \approx 1$. (The implied constant is absolute.)
\end{lemma}
Note that together with Lemma \ref{lem:dens}, this implies \eqref{eq:Uofnu} and finishes the proof of Theorem \ref{thm:anal-cap}.
\begin{proof}
    For any $x\in K_n$ and $j\in \{0,\dots n\}$ denote by $D_j(x)$ the unique disc $D_{\jj}$ with $|\jj|=j$ which contains $x.$ We begin by proving the upper bound on $c_\nu(x)$:
    \begin{multline*}
        c_\nu(x)^2 = \iint \frac{1}{R(x,y,z)^2}\, d\nu(y)d\nu(z)\\
        \le 2\sum_{j=0}^{n-1} \int_{D_{j}(x)\setminus D_{j+1}(x)}\int_{D_{j}(x)} \frac{1}{R(x,y,z)^2}\, d\nu(y)d\nu(z) + \int_{D_n(x)}\int_{D_n(x)}\frac{1}{R(x,y,z)^2}\, d\nu(y)d\nu(z)\\
        \eqcolon (I) + (II).
    \end{multline*}
    We estimate $(I)$ and $(II)$ separately. Using the elementary inequality $R(x,y,z)\ge |x-y|/2$, as well as the separation of discs noted in \eqref{eq:separation}, we get the upper bound on $I$ as follows:
    \begin{align*}
        (I) &\le 8\sum_{j=0}^{n-1} \int_{D_{j}(x)\setminus D_{j+1}(x)}\int_{D_{j}(x)} \frac{1}{|x-z|^2}\, d\nu(y)d\nu(z) 
        \\
        &\stackrel{\eqref{eq:separation}}{\lesssim} \sum_{j=0}^{n-1} \int_{D_{j}(x)\setminus D_{j+1}(x)}\int_{D_{j}(x)} \frac{1}{r_{j}^2}\, d\nu(y)d\nu(z) 
        \\
        &\le \sum_{j=0}^{n-1} \frac{\nu(D_j(x))^2}{r_{j}^2} = \sum_{j=0}^{n-1} \frac{A^{-2}4^{-2j}}{r_{j}^2} = A^{-2}\sum_{j=0}^{n-1}\theta_j^2.
    \end{align*}
    Concerning $(II)$, we use the symmetry of the integral and the inequality $R(x,y,z)\ge |x-y|/2$ to get
    \begin{multline*}
        (II)= 2\iint_{\substack{y,z\in D_n(x),\\ |x-y|\ge |x-z|}}\frac{1}{R(x,y,z)^2}\, d\nu(y)d\nu(z)\le  8\iint_{\substack{y,z\in D_n(x),\\ |x-y|\ge |x-z|}}\frac{1}{|x-y|^2}\, d\nu(y)d\nu(z)\\
        =\int_{y\in D_n(x)}\frac{\nu(B(x, |x-y|)}{|x-y|^2}\, d\nu(y) = 
        \frac{1}{A\mathcal{L}^2(K_n)}\int_{D_n(x)}\frac{\mathcal{L}^2(B(x, |x-y|)}{|x-y|^2}\, d\nu(y)\\
        \lesssim  \frac{\nu(D_n(x))}{A\mathcal{L}^2(K_n)} = \frac{4^{-n}}{ A^2 4^n \pi r_n^2} \approx A^{-2}\theta_n^2.
    \end{multline*}
   Putting together the estimates for $(I)$ and $(II)$ gives
    \begin{equation*}
        c_\nu(x)^2 \lesssim A^{-2}\sum_{j=0}^{n}\theta_j^2 = 1.
    \end{equation*}

    We move on to the lower bound on $c_\nu(x)$. We again start by dividing into the intermediate scales and the smallest scale:
    \begin{multline*}
        c_\nu(x)^2 = \iint \frac{1}{R(x,y,z)^2}\, d\nu(y)d\nu(z)\\
        \ge 
        \sum_{j=1}^{n}
        \int_{D_{j-1}(x)\setminus D_{j}(x)}
        \int_{D_{j-1}(x)\setminus D_{j}(x)} 
        \frac{1}{R(x,y,z)^2}\, d\nu(z)d\nu(y) + \int_{D_n(x)}\int_{D_n(x)}\frac{1}{R(x,y,z)^2}\, d\nu(y)d\nu(z)\\
        = (III) + (IV).
    \end{multline*}    
    To deal with $(III)$, fix $y \in D_{j-1}(x)\setminus D_{j}(x)$. Let $\wt D_j(x,y)$ denote one of the two discs of generation $j$ in $D_{j-1}(x)\setminus (D_{j}(x)\cup D_{j}(y))$ which maximizes the distance from the line passing through $x$ and $y$. For $z \in \wt D_j(x,y)$, we have $|x-y|\approx |x-z|\approx |y-z|\approx r_{j-1}$ by \eqref{eq:separation}. Furthermore, by elementary geometry, the distance from $z$ to the line through $x$ and $y$ is $\gtrsim r_{j-1}$, so $R(x,y,z) \approx r_{j-1}$. 
    We use this estimate to bound $(III)$:    
    \begin{multline*}
    (III)\ge \sum_{j=1}^{n}\int_{D_{j-1}(x)\setminus D_{j}(x)}\int_{\wt D_{j}(x,y)}\frac{1}{R(x,y,z)^2}\, d\nu(y)d\nu(z)
    \approx\sum_{j=1}^{n}\frac{4^{-2j}}{A^2r_{j-1}^2} \approx A^{-2}\sum_{j=0}^{n-1}\theta_j^2
    \end{multline*}
    
    To bound $(IV)$ we argue similarly. Let $y_0, z_0\in D_n(x)$ be such that $x, y_0, z_0$ are vertices of an equilateral triangle with sidelength $r_n/2,$ and such that $B(y_0, r_n/10)\subset D_n(x)$, $B(z_0, r_n/10)\subset D_n(x)$. As before, for all $y\in B(y_0, r_n/10),$ $z\in B(z_0, r_n/10),$ the distance from $z$ to the line through $x$ and $y$ is $\gtrsim r_{n}$, so $R(x,y,z) \approx r_{n}$,
    and so
    \begin{multline*}
        (IV)\ge \int_{B(y_0, r_n/10)}\int_{B(z_0, r_n/10)}\frac{1}{R(x,y,z)^2}\, d\nu(z)d\nu(y)\approx \frac{\nu(B(y_0, r_n/10))\nu(B(y_0, r_n/10))}{r_n^2}\\
        \approx \frac{A^{-2}\mathcal{L}^2(K_n)^{-2}r_n^4}{r_n^2}\approx A^{-2}4^{-2n}r_n^{-4}r_n^2= A^{-2}\theta_n^2.
    \end{multline*}
    Together with the estimate for $(III)$ this gives
    \begin{equation*}
        c_\nu(x)^2\gtrsim A^{-2}\sum_{j=0}^{n}\theta_j^2= 1.
    \end{equation*}
\end{proof}

\appendix

\section{Prescribing exact Favard length of a countable set}
\label{section:prescribe countable}

Let $N:(0,\infty)\to \mathbb{N}$ be a right-continuous, locally integrable, non-increasing function such that $N(r)=1$ for all sufficiently large $r$. Let $f(r):=2\int_0^r N(s)ds$. It can be easily shown that for every compact set $K$, the function $r\mapsto |K(r)|$ arises as an $f$ constructed as above. We now aim to show that this is a complete characterization.

\begin{proposition}
    Given $f$ as above, there exists a finite or countable set $K\subset \R$ such that $|K(r)|=f(r)$ for every $r>0$. 
\end{proposition}

\begin{proof}
    From the assumptions on $N$ we can find a decreasing  sequence of positive radii $r_j>0$ (we adopt the convention $r_0=\infty$) and an increasing sequence of natural numbers $n_j$ such that
    \begin{equation}\label{eq:definition_N}
    N(r)=n_j\qquad\text{if $r\in [r_{j+1},r_j) $.}
    \end{equation}
    In particular $n_0=1$. 
    We aim to construct $K$ by putting $(n_1-n_0)$-many points each at distance $2r_1$ from the previous one, followed by $(n_2-n_1)$-many points each at distance $2r_2$ from the previous one, and so on. 
    
    To formalize this, let $R$ be the ``inverse'' of $N$,  defined by
    \[
    R(t)=r_j\qquad\text{if $t\in [n_{j-1},n_j)$.}
    \]
    (The graph of $R$ is obtained by flipping the graph of $N$, including the filled-in vertical jumps, with respect to the bisector $\{t=r\}$). For $k\in\mathbb{N}_+$ define
    \[
    x_k:=2\int_1^k R(t)dt.
    \]
    We claim that $K:=\{x_k\}_{k\in \mathbb{N}_+}$ satisfies $|K(r)|=f(r)$. To this aim, it is sufficient to show that $K(r)$ has $N(r)$-many connected components. First, the distance between two consecutive points is
    \begin{equation}\label{eq:distance_consecutive}
    |x_{k+1}-x_k|=2\int_k^{k+1} R(t)dt=2r_j\qquad \text{for $k\in [n_{j-1},n_j)$.}
    \end{equation}
    Fix $r>0$; hence $r_{j+1}\le r<r_j$ for some $j$. Then all points up to and including those at distance $2r_j$ from the next one will belong to the same connected component  of $K(r)$; on the other hand, the points before those will each belong to their own connected component. Hence, using \eqref{eq:distance_consecutive} and \eqref{eq:definition_N}, the number of connected components of $K(r)$ is
    \[
    \#\{k\in \mathbb{N}_+:\, |x_{k+1}-x_k|>2r\}+1=\#\{k\in \mathbb{N}_+:\, k<n_j\}+1=n_j=N(r).\qedhere
    \]    
\end{proof}

\bibliographystyle{amsalpha}
\bibliography{ref}

\end{document}